\RequirePackage{fix-cm} % Fix LaTeX2e bugs.
\documentclass[a4paper, twoside, reqno, 12pt]{amsart}
\usepackage{fixltx2e}   % Fix LaTeX2e bugs.

\usepackage{etex}

\usepackage{lmodern}
\usepackage[latin1]{inputenc}
\usepackage[T1]{fontenc}

\usepackage[final, kerning=true, spacing=true, factor=1100, stretch=10, shrink=10]{microtype}

\usepackage{esint}
\usepackage{dsfont}
\usepackage{xspace}
\usepackage{amsgen}
\usepackage{amsthm}
\usepackage{amssymb}
\usepackage{amsmath}
\usepackage{wasysym}
\usepackage{upgreek}
\usepackage{amsfonts}
\usepackage{stmaryrd}
\usepackage{mathtools}
\usepackage{textgreek}

\usepackage[mathcal, mathscr]{euscript}

\usepackage{mathrsfs}
\DeclareMathAlphabet{\mathscrbf}{OMS}{mdugm}{b}{n}

\usepackage[ugly]{nicefrac}

\usepackage[lined, boxed, linesnumbered, english]{algorithm2e}

\usepackage{multicol}
\usepackage{multirow}

\usepackage{arydshln}
\usepackage{tabularx}
\usepackage{colortbl}
\usepackage{multicol}
\usepackage{multirow}

\usepackage{a4wide}

\usepackage[dvipsnames, table]{xcolor}
\definecolor{bckg}{RGB}{20.8, 20.8, 20.8}
\definecolor{oneblue}{rgb}{0.0, 0.0, 0.85}
\definecolor{Lightblue}{RGB}{214, 214, 214}
\definecolor{bluepigment}{rgb}{0.2, 0.2, 0.6}
\definecolor{charcoal}{rgb}{0.21, 0.27, 0.31}
\definecolor{denimblue}{rgb}{0.08, 0.38, 0.74}
\definecolor{Lightgray}{rgb}{0.89, 0.89, 0.89}
\definecolor{darkgrey}{rgb}{0.273, 0.281, 0.30}
\definecolor{darkelectricblue}{rgb}{0.33, 0.41, 0.47}

\usepackage[sort&compress, comma, square, numbers]{natbib}

\usepackage{psfrag}
\usepackage{graphicx}
\usepackage{subfigure}
\usepackage{morefloats}
\usepackage{indentfirst}

\usepackage{acronym}

\usepackage[labelsep=period,%
            labelfont={bf,sf,color=bluepigment},%
            textfont=it,%
            justification=raggedright]{caption}

\usepackage[perpage, symbol]{footmisc}

\usepackage[usenames, pdf]{pstricks}
\usepackage{epsfig}
\usepackage{pst-grad} % For gradients
\usepackage{pst-plot} % For axes

\usepackage{url}
\usepackage[colorlinks,
           urlcolor=oneblue,
           linkcolor=MidnightBlue,
           citecolor=RoyalBlue,
           bookmarksopen=false,
           pdfpagemode=UseNone,
           pagebackref]{hyperref}

\usepackage[explicit]{titlesec}

\titleformat{\section}[block]
  {\color{NavyBlue}\Large\sffamily\bfseries}
  {}
  {0.0em}
  {\colorbox{bckg!5}{\strut\parbox{\dimexpr\linewidth-2\fboxsep\relax}{\thesection. #1}}}
  [\vspace*{0.33em}]

\titleformat{name=\section,numberless}[block]
  {\color{NavyBlue}\Large\sffamily\bfseries}
  {}
  {0.0em}
  {\colorbox{bckg!5}{\strut\parbox{\dimexpr\linewidth-2\fboxsep\relax}{#1}}}
  [\vspace*{0.33em}]

\titleformat{\subsection}
  {\color{NavyBlue}\large\sffamily\bfseries}
  {}
  {0.0em}
  {\colorbox{bckg!5}{\parbox{\dimexpr\linewidth-2\fboxsep\relax}{\thesubsection. #1}}}
  [\vspace*{0.33em}]

\titleformat{name=\subsection,numberless}
  {\color{NavyBlue}\Large\sffamily\bfseries}
  {}
  {0em}
  {\colorbox{bckg!5}{\parbox{\dimexpr\linewidth-2\fboxsep\relax}{#1}}}
  [\vspace*{0.33em}]

\titleformat{\subsubsection}
  {\color{bluepigment}\sffamily\normalsize\bfseries}
  {\thesubsubsection}
  {0.5em}
  {#1}
  [\vspace*{0.33em}]

\titleformat{\paragraph}[runin]
  {\color{bluepigment}\sffamily\small\bfseries}
  {}
  {0em}
  {#1}

\titlespacing{\section}{0.0em}{1.5em plus 2pt minus 2pt}%
{1.0em plus 2pt minus 2pt}[0em]
\titlespacing{\subsection}{0.5em}{1.5em plus 2pt minus 2pt}%
{1.0em}[0em]
\titlespacing{\subsubsection}{0.5em}{1.5em plus 2pt minus 2pt}%
{1.0em plus 2pt minus 2pt}[0em]

\usepackage{titletoc}

\contentsmargin{0.5em}
\newlength{\tocsep} 
\titlecontents{section}[\tocsep]
  {\addvspace{10pt}\bfseries\sffamily}
  {\contentslabel[\thecontentslabel]{\tocsep}}
  {}
  {\ \titlerule*[0.75pc]{.}\ \thecontentspage}
  []
\titlecontents{subsection}[\tocsep]
  {\addvspace{8pt}\sffamily}
  {\contentslabel[\thecontentslabel]{\tocsep}}
  {}
  {\ \titlerule*[0.5pc]{.}\ \thecontentspage}
  []
\titlecontents*{subsubsection}[\tocsep]
  {\addvspace{2pt}\footnotesize\sffamily}
  {}
  {}
  {\ \titlerule*[0.35pc]{.}\ \thecontentspage}
  [\\*]

\makeatletter
\def\@setauthors{%
  \begingroup
  \def\thanks{\protect\thanks@warning}%
  \trivlist
  \centering\footnotesize \@topsep30\p@\relax
  \advance\@topsep by -\baselineskip
  \item\relax
  \author@andify\authors
  \def\\{\protect\linebreak}%
  \textsc{\normalsize\textcolor{charcoal}{\authors}}%
  \ifx\@empty\contribs
  \else
    ,\penalty-3 \space \@setcontribs
    \@closetoccontribs
  \fi
  \endtrivlist
  \endgroup
}
\def\@settitle{\begin{center}%
  \baselineskip14\p@\relax
    \bfseries
    \textsc{\Large\textcolor{charcoal}{\@title}}
  \end{center}%
}
\makeatother

\usepackage{enumitem}
\setlist[description]{%
  topsep=30pt,               % space before start / after end of list
  itemsep=5pt,               % space between items
  font={\bfseries\sffamily\color{NavyBlue}}, % if colour is needed
}

\usepackage{fancyhdr}
\usepackage{lastpage}

\newcommand*\Title{\textcolor{bluepigment}{On the multi-symplectic structure of Boussinesq-type equations. II}}
\newcommand*\Authors{\textcolor{bluepigment}{A.~Dur\'an, D.~Dutykh \& D.~Mitsotakis}}
\newcommand*{\plogo}{\textcolor{gray}{{\texttt{arXiv.org} / \textsc{hal}}}} % Generic publisher logo

\numberwithin{equation}{section}

\newtheorem{remark}{Remark}

\newtheorem{theorem}{Theorem}

\newcommand{\up}[1]{$^{\mathrm{\small\textsf{#1}}}$} % \up command from French babel

\newcommand{\N}{\mathds{N}}
\newcommand{\R}{\mathds{R}}
\newcommand{\Z}{\mathds{Z}}
\newcommand{\ui}{\mathrm{i}}
\newcommand{\ue}{\mathrm{e}}

\renewcommand{\O}{\mathcal{O}}
\renewcommand{\nu}{\text{\textnu}}

\renewcommand{\eta}{\text{\texteta}}
\renewcommand{\beta}{\text{\textbeta}}
\renewcommand{\mu}{\text{\textmugreek}}
\renewcommand{\alpha}{\text{\textalpha}}
\renewcommand{\kappa}{\text{\textkappa}}
\renewcommand{\omega}{\text{\textomega}}
\renewcommand{\theta}{\text{\texttheta}}
\newcommand{\ud}{\mathrm{d}\hspace{0.08em}}

\newcommand{\ep}{\mathfrak{e}}
\newcommand{\J}{\mathds{J}}
\newcommand{\K}{\mathds{K}}
\newcommand{\M}{\mathds{M}}

\newcommand{\Zz}{\mathbf{Z}}
\newcommand{\Id}{\mathds{I}}

\newcommand{\A}{\mathfrak{A}}
\newcommand{\B}{\mathfrak{B}}
\newcommand{\E}{\mathfrak{E}}
\newcommand{\F}{\mathfrak{F}}
\newcommand{\I}{\mathfrak{I}}
\newcommand{\Ec}{\mathcal{E}}
\newcommand{\Ic}{\mathcal{I}}
\newcommand{\Gg}{\mathfrak{G}}
\newcommand{\Gm}{\mathfrak{M}}

\renewcommand{\H}{\mathscr{H}}
\renewcommand{\geq}{\geqslant}
\renewcommand{\leq}{\leqslant}
\newcommand{\z}{\boldsymbol{z}}
\renewcommand{\S}{\mathfrak{S}}
\newcommand{\Mm}{\mathcal{M}}
\newcommand{\Dd}{\mathcal{D}}
\newcommand{\Cc}{\mathcal{C}}
\newcommand{\Nn}{\mathcal{N}}

\newcommand{\D}{\mathfrak{D}}

\renewcommand{\L}{\mathds{L}}

\renewcommand{\mapsto}{\longmapsto}

\newcommand{\cf}{\emph{cf.}\xspace}
\newcommand{\ie}{\emph{i.e.}\xspace}
\newcommand{\eg}{\emph{e.g.}\xspace}

\newcommand{\etal}{\emph{et al.}\xspace}

\newcommand{\scal}{\boldsymbol{\cdot}}
\newcommand{\grad}{\boldsymbol{\nabla}}

\newcommand{\od}[2]{\frac{\mathrm{d}\, #1}{\mathrm{d}\/#2}}

\newcommand{\Prod}[2]{\left\langle\, #1\,,\,#2\,\right\rangle}

\newcommand{\eqdef}{\mathop{\stackrel{\,\mathrm{def}}{:=}\,}}

\newcommand{\vdd}[2]{\dfrac{\delta #1}{\delta\hspace{0.0556em} #2}}

\newcommand{\half}{{\textstyle{1\over2}}}
\newcommand{\third}{{\textstyle{1\over3}}}

\makeatletter
\renewcommand*\env@matrix[1][\arraystretch]{%
  \edef\arraystretch{#1}%
  \hskip -\arraycolsep
  \let\@ifnextchar\new@ifnextchar
  \array{*\c@MaxMatrixCols c}}
\makeatother

\usepackage{acronym}
\acrodef{bvp}[BVP]{Boundary Value Problem}
\acrodef{NSWE}{Nonlinear Shallow Water Equations}

\begin{document}

\title[\Title]{On the multi-symplectic structure of Boussinesq-type systems. II: Geometric discretization}

\author[A.~Dur\'an]{Angel Dur\'an}
\address{\textbf{\textcolor{MidnightBlue}{A.~Dur\'an:}} Departamento de Matem\'atica Aplicada, E.T.S.I. Telecomunicaci\'on, Campus Miguel Delibes, Universidad de Valladolid, Paseo de Belen 15, 47011 Valladolid, Spain}
\email{angel@mac.uva.es\vspace*{0.5em}}

\author[D.~Dutykh]{Denys Dutykh$^*$}
\address{\textbf{\textcolor{MidnightBlue}{D.~Dutykh:}} Univ. Grenoble Alpes, Univ. Savoie Mont Blanc, CNRS, LAMA, 73000 Chamb\'ery, France and LAMA, UMR 5127 CNRS, Universit\'e Savoie Mont Blanc, Campus Scientifique, F-73376 Le Bourget-du-Lac Cedex, France}
\email{Denys.Dutykh@univ-smb.fr}
\urladdr{http://www.denys-dutykh.com/\vspace*{0.5em}}
\thanks{$^*$ Corresponding author}

\author[D.~Mitsotakis]{Dimitrios Mitsotakis}
\address{\textbf{\textcolor{MidnightBlue}{D.~Mitsotakis:}} Victoria University of Wellington, School of Mathematics and Statistics, PO Box 600, Wellington 6140, New Zealand}
\email{dmitsot@gmail.com}
\urladdr{http://dmitsot.googlepages.com/}


%%% ------------------------------------------------------------------------ %%%

\begin{titlepage}
\clearpage
\pagenumbering{arabic}
\thispagestyle{empty} % Remove page numbering on this page
\noindent
{\Large Angel \textsc{Dur\'an}}\\
{\it\textcolor{gray}{Universidad de Valladolid, Spain}}
\\[0.02\textheight]
{\Large Denys \textsc{Dutykh}}\\
{\it\textcolor{gray}{CNRS, Universit\'e Savoie Mont Blanc, France}}
\\[0.02\textheight]
{\Large Dimitrios \textsc{Mitsotakis}}\\
{\it\textcolor{gray}{Victoria University of Wellington, New Zealand}}
\\[0.16\textheight]

\vspace*{0.99cm}

\colorbox{Lightblue}{
  \parbox[t]{1.0\textwidth}{
    \centering\huge\sc
    \vspace*{0.75cm}
    
    \textcolor{bluepigment}{On the multi-symplectic structure of Boussinesq-type systems. II: Geometric discretization}
    
    \vspace*{0.75cm}
  }
}

\vfill % Whitespace between the title block and the publisher

\raggedleft     % Right-align all text
{\large \plogo} % Publisher and logo
\end{titlepage}

%%% ------------------------------------------------------------------------ %%%

\thispagestyle{empty} % Remove page numbering on this page
\par\vspace*{\fill}   % Whitespace until the bottom
\begin{flushright} % Right-align all text
{\textcolor{denimblue}{\textsc{Last modified:}} \today}
\end{flushright}

%%% ------------------------------------------------------------------------ %%%

\begin{abstract}

In this paper we consider the numerical approximation of systems of \textsc{Boussinesq}-type to model surface wave propagation. Some theoretical properties of these systems (multi-symplectic and \textsc{Hamiltonian} formulations, well-posedness and existence of solitary-wave solutions) were previously analyzed by the authors in Part I. As a second part of the study, considered here is the construction of geometric schemes for the numerical integration. By using the method of lines, the geometric properties, based on the multi-symplectic and \textsc{Hamiltonian} structures, of different strategies for the spatial and time discretizations are discussed and illustrated.

% In this paper we consider the numerical approximation of systems of Boussinesq-type to model surface wave propagation. Some theoretical properties of these systems (multi-symplectic and Hamiltonian formulations, well-posedness and existence of solitary-wave solutions) were previously analyzed by the authors in Part I. As a second part of the study, considered here is the construction of geometric schemes for the numerical integration. By using the method of lines, the geometric properties, based on the multi-symplectic and Hamiltonian structures, of different strategies for the spatial and time discretizations are discussed and illustrated.

\bigskip\bigskip
\noindent \textbf{\keywordsname:} geometric numerical integration; symplectic methods; multi-symplectic schemes; Boussinesq equations; surface waves \\

\smallskip \\
\noindent \textbf{PACS [2010]:} 47.35.Bb (primary), 47.35.Fg (secondary)

\end{abstract}

%%% ----------------------------------------------------------------------- %%%

\newpage
\maketitle
\thispagestyle{empty}

%%% ----------------------------------------------------------------------- %%%

\clearpage
\thispagestyle{empty}
\tableofcontents

%%% ----------------------------------------------------------------------- %%%

\clearpage
\vspace*{0.25em}
\section{Introduction}

In a previous paper, \cite{Duran2019}, the authors considered systems of \textsc{Boussinesq}-type in the form
\begin{eqnarray}
  \eta_{\,t}\ +\ [\,u\ +\ \A\,(\eta\,,\,u)\ +\ a\,u_{\,x\,x}\ -\ b\,\eta_{\,x\,t}\,]_{\,x}\ &=&\ 0\,, \label{eq:5s} \\
  u_{\,t}\ +\ \bigl[\,\eta\ +\ \B\,(\eta\,,\,u)\ +\ c\,\eta_{\,x\,x}\ -\ d\,u_{\,x\,t}\,\bigr]_{\,x}\ &=&\ 0\,, \label{eq:6s}
\end{eqnarray}
with constant parameters $a\,$, $b\,$, $c\,$, $d$ and homogeneous, quadratic nonlinearities
\begin{align*}
  \A\,(\eta\,,\,u)\ &\eqdef\ \alpha_{\,1\,1}\,\eta^{\,2}\ +\ \alpha_{\,1\,2}\,\eta\,u\ +\ \alpha_{\,2\,2}\,u^{\,2}\,, \\
  \B\,(\eta\,,\,u)\ &\eqdef\ \beta_{\,1\,1}\,\eta^{\,2}\ +\ \beta_{\,1\,2}\,\eta\,u\ +\ \beta_{\,2\,2}\,u^{\,2}\,,
\end{align*}
with real coefficients $\alpha_{\,\imath\,\jmath}\,$, $\beta_{\,\imath\,\jmath}$. The parameters $a\,$, $b\,$, $c\,$, $d$ can be defined as, \cite{BCS, Bona2004}
\begin{align*}
  a\ &\eqdef\ \half\;\bigl(\theta^{\,2}\ -\ \third\bigr)\,\nu\,, \qquad
  & b\ \eqdef\ \half\;\bigl(\theta^{\,2}\ -\ \third\bigr)\scal(1\ -\ \nu)\,, \\
  c\ &\eqdef\ \half\;\bigl(1\ -\ \theta^{\,2}\bigr)\,\mu\,, \qquad
  & d\ \eqdef\ \half\;\bigl(1\ -\ \theta^{\,2}\bigr)\scal(1\ -\ \mu)\,,
\end{align*}
where $\nu\,$, $\mu\ \in\ \R$ and $0\ \leq\ \theta\ \leq\ 1\,$. The family of Systems \eqref{eq:5s}, \eqref{eq:6s} can be proposed as approximation to the full \textsc{Euler} equations for surface wave propagation along horizontal, straight channel in the case of long, small amplitude waves and a \textsc{Stokes}--\textsc{Ursell} number \cite{Ursell1953} of {\bf order unity}. The real-valued functions $\eta\ =\ \eta\,(x,\,t)$ and $u\ =\ u\,(x,\,t)$ represent, respectively, the deviation of the free surface at the position $x$ along the channel and time $t$ and the horizontal velocity at $(x,\,t)$ at a nondimensional height $y\ =\ -\,1\ +\ \theta\,(1\ +\ \eta\,(x,\,t))\,$, see \textsc{Bona} \etal \cite{BCS, Bona2004}.

One of the highlights of \cite{Duran2019} is the derivation of models of the form \eqref{eq:5s} with multi-symplectic (MS) structure. A system of partial differential equations (PDEs) is said to be multi-symplectic in one space dimension if it can be written in the form
\begin{equation}\label{eq:ms}
  \K\scal\z_{\,t}\ +\ \M\scal\z_{\,x}\ =\ \grad_{\,\z}\,\S\,(\z)\,, \qquad \z\ \in\ \R^{\,d}\,,
\end{equation}
for some $d\ \geq\ 3\,$, where  $\z\,(x,\,t)\,:\ \R\times\R^{\,+}\ \mapsto\ \R^{\,d}\,$, $\K$ and $\M$ are real, skew-symmetric $d\times d$ matrices, the dot $\scal$ denotes the matrix-vector product in $\R^{\,d}\,$, $\grad_{\,\z}$ is the classical gradient operator in $\R^{\,d}$ and the potential functional $\S\,(\z)\,$ is assumed to be a smooth function of $\z\,$.

The multi-symplectic theory generalizes the classical \textsc{Hamiltonian} formulations, \cite{Basdevant2007}, to the case of PDEs such that the space and time variables are treated on an equal footing. For more details related to the fundamentals of the theory and applications of MS systems we refer to \cite{Bridges1997}. One of the main features of the multi-symplectic formulation is the existence of a multi-symplectic conservation law
\begin{equation}\label{eq:cons}
  \omega_{\,t}\ +\ \kappa_{\,x}\ =\ 0\,,
\end{equation}
where
\begin{equation}\label{eq:def}
  \omega\ \eqdef\ \frac{1}{2}\;\ud\z\,\wedge\,(\K\scal\ud\z)\,, \qquad
  \kappa\ \eqdef\ \frac{1}{2}\;\ud\z\,\wedge\,(\M\scal\ud\z)\,,
\end{equation}
with $\wedge$ being the standard exterior product of differential forms, \cite{Spivak1971}. Note that condition \eqref{eq:def} is local, it does not depend on specific boundary conditions. Additionally, when the function $\S\,(\z)$ does not depend explicitly on $x$ or $t\,$, then local energy and momentum are preserved:
\begin{eqnarray}
  \E_{\,t}\ +\ \F_{\,x}\ &=\ 0\,, \label{eq:cons2a} \\
  \I_{\,t}\ +\ \Gm_{\,x}\ &=\ 0\,. \label{eq:cons2b}
\end{eqnarray}
Defining the generalized energy $\E$ and generalized momentum $\I$ densities as
\begin{equation*}
 \E\,(\z)\ \eqdef\ \S\,(\z)\ -\ \half\;\Prod{\z}{\M\scal\z_{\,x}}\,, \qquad
 \I\,(\z)\ \eqdef\ \half\;\Prod{\z}{\K\scal\z_{\,x}}\,,
\end{equation*}
and corresponding fluxes
\begin{equation*}
 \F\,(\z)\ \eqdef\ \half\;\Prod{\z}{\M\scal\z_{\,t}}\,, \qquad
 \Gm\,(\z)\ \eqdef\ \S\,(\z)\ -\ \half\;\Prod{\z}{\K\scal\z_{\,t}}\,.
\end{equation*}
The symbol $\Prod{\scal}{\scal}$ denotes the standard scalar product in $\R^{\,d}\,$. Throughout the paper, $\Id_{\,n}$ will stand for the $n\times n$ identity matrix.

For the specific case of the present study, it was shown in \cite{Duran2019} that when
\begin{equation}\label{eq:const}
  a\ =\ c\,, \qquad \alpha_{\,1\,2}\ =\ 2\,\beta_{\,1\,1}\,, \qquad \beta_{\,1\,2}\ =\ 2\,\alpha_{\,2\,2}\,,
\end{equation}
then Equations \eqref{eq:5s}, \eqref{eq:6s} can be formulated in the MS form \eqref{eq:ms} for $d\ =\ 10\,$,
\begin{equation}\label{eq:K}
  \K\ =\ \begin{pmatrix}[1.1]
    0 & \half & -\half\;b & 0 & 0 & 0 & 0 & 0 & 0 & 0 \\
    -\half & 0 & 0 & 0 & 0 & 0 & 0 & 0 & 0 & 0 \\
    \half\;b & 0 & 0 & 0 & 0 & 0 & 0 & 0 & 0 & 0 \\
    0 & 0 & 0 & 0 & 0 & 0 & 0 & 0 & 0 & 0 \\
    0 & 0 & 0 & 0 & 0 & 0 & 0 & 0 & 0 & 0 \\
    0 & 0 & 0 & 0 & 0 & 0 & \half & -\half\;d & 0 & 0 \\
    0 & 0 & 0 & 0 & 0 & -\half & 0 & 0 & 0 & 0 \\
    0 & 0 & 0 & 0 & 0 & \half\;d & 0 & 0 & 0 & 0 \\
    0 & 0 & 0 & 0 & 0 & 0 & 0 & 0 & 0 & 0 \\
    0 & 0 & 0 & 0 & 0 & 0 & 0 & 0 & 0 & 0
  \end{pmatrix}\,,
\end{equation}
\begin{equation}\label{eq:M}
  \M\ =\ \begin{pmatrix}[1.1]
    0 & 0 & 0 & -\half\;b & 0 & 0 & 0 & a & 0 & 0 \\
    0 & 0 & 0 & 0 & -1 & 0 & 0 & 0 & 0 & 0 \\
    0 & 0 & 0 & 0 & 0 & -c & 0 & 0 & 0 & 0 \\
    \half\;b & 0 & 0 & 0 & 0 & 0 & 0 & 0 & 0 & 0 \\
    0 & 1 & 0 & 0 & 0 & 0 & 0 & 0 & 0 & 0 \\
    0 & 0 & c & 0 & 0 & 0 & 0 & 0 & -\half\;d & 0 \\
    0 & 0 & 0 & 0 & 0 & 0 & 0 & 0 & 0 & -1 \\
    -a & 0 & 0 & 0 & 0 & 0 & 0 & 0 & 0 & 0 \\
    0 & 0 & 0 & 0 & 0 & \half\;d & 0 & 0 & 0 & 0 \\
    0 & 0 & 0 & 0 & 0 & 0 & 1 & 0 & 0 & 0
  \end{pmatrix}\,,
\end{equation}
and
\begin{multline}\label{eq:pot}
  \S\,(\z)\ \eqdef\ p_{\,1}\,\eta\ -\ \eta\,u\ -\ \third\;\alpha_{\,1\,1}\,\eta^{\,3}\ -\ \beta_{\,1\,1}\,\eta^{\,2}\,u\ -\ \half\;\beta_{\,1\,2}\,\eta\,u^{\,2}\ +\ \half\;b\,v_{\,1}\,w_{\,1} \\ 
  -\ \third\;\beta_{\,2\,2}\,u^{\,3}\ +\ \half\;d\,v_{\,2}\,w_{\,2}\ -\ a\,v_{\,1}\,v_{\,2}\ +\ p_{\,2}\,u\,.
\end{multline}
where
\begin{equation*}\label{eq:zdef}
  \z\ \eqdef\ \bigl(\,\eta\,,\,\phi_{\,1},\,v_{\,1},\,w_{\,1},\,p_{\,1},\,u,\,\phi_{\,2},\,v_{\,2},\,w_{\,2},\,p_{\,2}\,\bigr)\ \in\ \R^{\,10}\,.
\end{equation*}

Furthermore, when
\begin{equation}\label{eq:const3b}
  \beta_{\,1\,2}\ =\ 2\,\alpha_{\,1\,1}\ =\ 2\,\alpha_{\,2\,2}\,, \qquad
  \alpha_{\,1\,2}\ =\ 2\,\beta_{\,1\,1}\ =\ 2\,\beta_{\,2\,2}\,,
\end{equation}
then the corresponding $(a,\,b,\,a,\,b)$ System~\eqref{eq:5s}, \eqref{eq:6s} is multi-symplectic \emph{and} \textsc{Hamiltonian}:
\begin{equation}\label{eq:ham}
  \begin{pmatrix}
    \eta_{\,t} \\
    u_{\,t}
  \end{pmatrix}\ =\ 
  \J\scal\begin{pmatrix}[2]
    \vdd{\H}{\eta} \\
    \vdd{\H}{u}
 \end{pmatrix}\,,
\end{equation}
on a suitable functional space for $(\eta,\, u)\,$, where the co-symplectic matrix operator $\J$ is given by
\begin{equation}\label{eq:sympl}
  \J\ \eqdef\ \begin{pmatrix}
  0 & -\,\bigl(1\ -\ b\,\partial_{\,x\,x}^{\,2}\bigr)^{\,-\,1}\circ\partial_{\,x} \\
  -\,\bigl(1\ -\ b\,\partial_{\,x\,x}^{\,2}\bigr)^{\,-\,1}\circ\partial_{\,x} & 0
  \end{pmatrix}
\end{equation}
and $\Bigl(\vdd{\H}{\eta},\, \vdd{\H}{u}\Bigr)^{\top}$ denotes {\bf the variational derivative, \cite{Olver1993}}, and the \textsc{Hamiltonian} functional is defined as
\begin{equation}\label{eq:hams}
  \H\ \eqdef\ \frac{1}{2}\;\int_{\,\R}\bigl\{\,\eta^{\,2}\ +\ u^{\,2}\ -\ a\,(\eta_{\,x}^{\,2}\ +\ u_{\,x}^{\,2})\ +\ 2\,\Gg\,(\eta\,,\,u)\,\bigr\}\;\ud x\,,
\end{equation}
with
\begin{equation}\label{eq:gham}
  \Gg\,(\eta\,,\,u)\ \eqdef\ \frac{\beta_{\,1\,1}}{3}\;\eta^{\,3}\ +\ \frac{\beta_{\,1\,2}}{2}\;\eta^{\,2}\,u\ +\ \beta_{\,1\,1}\,\eta\,u^{\,2}\ +\ \frac{\beta_{\,1\,2}}{6}\;u^{\,3}\,.
\end{equation}
Additional conserved quantities are
\begin{equation}\label{eq:quadraticq}
  \I\ \eqdef\ \;\int_{\,\R}\bigl\{\,\eta\,u\ +\ b\,\eta_{\,x}\,u_{\,x} \,\bigr\}\;\ud x\,,
\end{equation}
\begin{equation}\label{eq:linearq}
  {C}_{\,1}(\eta,\,u)\ \eqdef\ \;\int_{\,\R}\,\eta\;\ud x\,, \qquad {C}_{\,2}(\eta,u)\ \eqdef\ \;\int_{\,\R}\,u\;\ud x\,.
\end{equation}

Similarly to the symplectic integration to approximate \textsc{Hamiltonian} ordinary differential equations, \cite{Sanz-Serna1994}, the construction and analysis of multi-symplectic methods for PDEs with MS structure appeared as a natural way to include the geometric numerical integration, \cite{Hairer2002}, as essential part of the numerical treatment of the multi-symplectic theory,  \cite{Bridges2001, Moore2003a, Chen2011, Dutykh2013a}. As mentioned in \cite{Reich2000}, a first numerical approach of MS type was derived by \textsc{Marsden} \etal \cite{Marsden1998}, by discretizing the \textsc{Lagrangian} of the \textsc{Cartan} form in field theory which appears in the MS structure for wave equations. A new definition is due to \textsc{Bridges} \& \textsc{Reich}, \cite{Bridges2001}, who consider a numerical scheme as multi-symplectic if it preserves a discrete version of the corresponding conservation law \eqref{eq:cons}. This definition agrees with that in \cite{Marsden1998} when the governing equations are given by a first-order \textsc{Lagrangian}.

Several properties of these geometric integrators have been pointed out (see \eg \cite{Bridges2006} and references therein). The most remarkable may be their behaviour with respect to local conservation laws, the preservation, in some cases and in the sense described in \cite{Ascher2004}, of a discrete version of the dispersion relation of the corresponding PDE and, finally, the development of a Backward Error Analysis (BEA) for MS integrators, \cite{Moore2003}. In particular, the numerical dispersion relations of some multi-symplectic methods do not contain spurious roots and preserve the sign of the group velocity, \cf \cite{Frank2006}.

The literature includes the construction of MS schemes for nonlinear dispersive equations like, among others, the celebrated \textsc{Korteweg-de Vries} (KdV) equation, the Nonlinear \textsc{Schr\"{o}dinger} (NLS) equation, the \textsc{Zakharov}--\textsc{Kuznetsov} and shallow water equations and the `good' scalar \textsc{Boussinesq} equation (see \cite{Bridges2006} for more cases). These references seem to distinguish two lines of research for constructing multi-symplectic integrators. The first one (see \eg \cite{Ascher2004, Ascher2005, Bridges2001, Dutykh2013a, Sun2004, Li2013, Huang2003}) is based on applying symplectic integrators in time and in space (not necessarily the same). This seems to be the natural discretization of the MS property, which is translated to a discrete MS conservation law and some other consequences, where the time and space variables are treated, as in the continuous case, on equal footing. As mentioned in \cite{McLachlan2014}, this approach generates several drawbacks that must be overcome. The main ones concern the inclusion of boundary conditions (which breaks the initially equal treatment of the independent variables) and the proper definition of a numerical method after the symplectic discretization. These drawbacks are particularly relevant for traditional families of symplectic integrators like the \textsc{Gau\ss}--\textsc{Legendre} \textsc{Runge}--\textsc{Kutta} (GLRK) methods, \cf \cite{Hairer2002}. More recently, the use of the \textsc{Lobatto} IIIA--IIIB Partitioned \textsc{Runge}--\textsc{Kutta} (PRK) methods to multi-symplectic systems was studied in \cite{Ryland2008}: sufficient conditions on the system for a spatial discretization with these methods to lead to explicit semi-discrete system of Ordinary Differential Equations (ODE) are discussed, and different properties (multi-symplectic character, dispersion analysis and global error) are studied. The PRK semi-discretization avoids the singularity in cases like the nonlinear wave equation, the NLS equation or the `good' \textsc{Boussinesq} equation \cite{Ryland2008}, and  the dispersion relation does not contain spurious waves, although only covers a discontinuous part of the continuous frequency.

A second approach in constructing MS integrators may be represented \eg by the references \cite{Bridges2001a, Chen2001, Islas2003a, Islas2006}, where a discrete multi-symplectic property of a discretization based on Fourier pseudospectral approximation in space and a symplectic time integration is analyzed in different PDEs with MS structure. This approach pays attention to the boundary conditions (of periodic type in the cases treated in the previous references), showing how they force to treat time and space variables in a different way. The meaning of the MS character is also adapted through corresponding discrete MS conservation laws, local conserved quantities and dispersion analysis.

These two approaches inspired the work developed in the present paper, whose main results are now highlighted:
\begin{itemize}
  
  \item The paper introduces the numerical approximation to the periodic Initial-Value Problem (IVP) of \eqref{eq:5s}, \eqref{eq:6s} from the point of view of the geometric integration. The choice of this type of boundary conditions is motivated by mainly two reasons: their typical use in the literature when implementing MS integrators (see \eg \cite{Bridges2001a, Reich2000a}) and the experimental context where the \textsc{Boussinesq} equations \eqref{eq:5s}, \eqref{eq:6s} may be considered (\eg dynamics of waves or comparisons with the \textsc{Euler} system, see \cite{Duran2019}) and for which periodic boundary conditions are usually imposed in the simulations. A first task treated here is the extension of the multi-symplectic and \textsc{Hamiltonian} structures of the initial-value problem for the corresponding families of Systems \eqref{eq:5s}, \eqref{eq:6s}. In the first case, since symplecticity is a local concept, the MS conservation law does not depend on specific boundary conditions, \cite{Reich2000a}, but other related properties, such as the preservation of the total symplecticity or the global energy and momentum, do and this is first studied for the case of the \textsc{Boussinesq}-type systems.

  \item The paper will then discuss the properties and drawbacks presented when the first approach in constructing MS integrators, above mentioned, is applied to the periodic IVP for \eqref{eq:5s}, \eqref{eq:6s}. Different properties will be illustrated by considering the approximation given by the Implicit Midpoint Rule (IMR). Among them, the requirement, proved in \cite{McLachlan2014}, of odd number of discretization points when the MS system is discretized in space by a GLRK method, in order to have a well-defined semi-discrete ODE system, is emphasized. The use of PRK methods is also studied. Here the MS formulation of \eqref{eq:5s}, \eqref{eq:6s} will be shown to not satisfy the conditions derived in \cite{Ryland2008} (\ie conditions that imply that the resulting ODE system from the spatial discretization with the \textsc{Lobatto} IIIA--IIIB methods is well defined). Several discretizations with the two-stage scheme of the family and different partitions of System \eqref{eq:5s}, \eqref{eq:6s} were still tried and they required similar conditions to those of the GLRK methods, in order to avoid the singularity.

  \item In view of these drawbacks, the second alternative is analyzed. Our approach here consists of discretizing  the MS formulation of \eqref{eq:5s}, \eqref{eq:6s} with a general grid operator approximating the first partial derivative in space and a symplectic integrator in time. Conditions on the spatial grid operator in order to obtain a well-defined semi-discrete ODE system are derived, along with a semi-discrete version of the MS conservation law, the preservation of semi-discrete energy and momentum, a linear dispersion analysis and the preservation of the \textsc{Hamiltonian} structure in the corresponding cases. We observe that some of these results are not exclusive of Equations~\eqref{eq:5s}, \eqref{eq:6s}; they can be applied to a general MS System~\eqref{eq:ms} with periodic boundary conditions and the approach with a general grid operator generalizes the particular case of the pseudospectral discretization typically used in the literature, \cite{Bridges2001a, Chen2001, Islas2003a,Islas2006}. Then the time discretization of the semi-discrete system with a symplectic method is shown to give a corresponding fully discrete conservation law, a numerical dispersion relation and the preservation of global quantities in the \textsc{Hamiltonian} cases. These results will be illustrated with the implicit midpoint rule and several remarks on the generalization to other symplectic integrators will be made.

\end{itemize}

The paper is structured as follows. Section~\ref{sec:sec2} is devoted to the extension of the MS and \textsc{Hamiltonian} structures to the periodic IVP for \eqref{eq:5s}, \eqref{eq:6s}. In Section~\ref{sec:sec3} some properties and drawbacks derived from the discretization with GLRK methods and \textsc{Lobatto} IIIA--IIIB PRK methods are described and illustrated. The alternative of a different numerical treatment of the independent variables, in the sense described above, is analyzed in Section~\ref{sec:sec4}. Conclusions are outlined in Section~\ref{sec:sec5}.

%%% ----------------------------------------------------------------------- %%%

\section{On the periodic initial-value problem}
\label{sec:sec2}

We first study the influence of the imposition of periodic boundary conditions on the dependent variables and their derivatives on the possible multi-symplectic and \textsc{Hamiltonian} structures of Equations~\eqref{eq:5s}, \eqref{eq:6s}. In the case of the MS formulation \eqref{eq:K}, \eqref{eq:M}, \eqref{eq:pot}, recall that the MS conservation law \eqref{eq:cons}, \eqref{eq:def} does not depend on specific boundary conditions. On the other hand, integrating \eqref{eq:cons2a}, \eqref{eq:cons2b} on a period interval $]\,0,\,L\,[\,$, periodic boundary conditions imply the preservation of global energy and momentum
\begin{equation}\label{eq:globalq}
  \Ec_{\,L}\ \eqdef\ \int_{\,0}^{\,L}\,\E\,(\z)\:\ud x\,, \qquad 
  \Ic_{\,L}\ \eqdef\ \int_{\,0}^{\,L}\,\I\,(\z)\:\ud x\,.
\end{equation}
These can be written in terms only of $\eta$ and $u$ as
\begin{align}
  \Ec_{\,L}\ &\eqdef \ \int_{\,0}^{\,L}\,\bigl(-\,\eta\,u\ +\ a\,\eta_{\,x}\,u_{\,x}\ -\ \S_{\,L}\,(\eta,\,u)\bigr)\:\ud x\,, \label{eq:globalq1} \\
  \Ic_{\,L}\ &\eqdef \frac{1}{2}\;\int_{\,0}^{\,L}\,\bigl(\eta^{\,2}\ +\ u^{\,2}\ +\ b\,\eta_{\,x}^{\,2}\ +\ d\,u_{\,x}^{\,2}\bigr)\:\ud x\,, \label{eq:globalq2}
\end{align}
where
\begin{equation}\label{eq:globalq3}
  \S_{\,L}\,(\eta,\,u)\ \eqdef \ \third\;\alpha_{\,1\,1}\,\eta^{\,3}\ +\ \beta_{\,1\,1}\,\eta^{\,2}\,u\ +\ \half\;\beta_{\,1\,2}\,\eta\,u^{\,2}\ +\ \third\;\beta_{\,2\,2}\,u^{\,3}\,.
\end{equation}

Note also that the integration of \eqref{eq:cons} on $]\,0,\,L\,[$ and the periodic boundary conditions imply the preservation of the total symplecticity $\partial_{\,t}\bar{\omega}\ \equiv\ 0$ where
\begin{equation}\label{eq:tsymp}
  \bar{\omega}\ \eqdef\ \int_{\,0}^{\,L}\,\omega\:\ud x\,.
\end{equation}
As in the case of the IVP, when \eqref{eq:const3b} holds, the periodic IVP of the $(a,\,b,\,a,\,b)$ System~\eqref{eq:5s}, \eqref{eq:6s} admits a \textsc{Hamiltonian} structure \eqref{eq:ham} with respect to \eqref{eq:sympl} and \textsc{Hamiltonian} function
\begin{equation}\label{eq:globalq4}
  \H_{\:L}\ \eqdef\ \frac{1}{2}\;\int_{\,0}^{\,L}\bigl\{\,\eta^{\,2}\ +\ u^{\,2}\ -\ a\,(\eta_{\,x}^{\,2}\ +\ u_{\,x}^{\,2})\ +\ 2\,\Gg\,(\eta\,,\,u)\,\bigr\}\;\ud x\,,
\end{equation}
with $\Gg$ given by \eqref{eq:gham}. The corresponding versions of the invariants \eqref{eq:quadraticq} and \eqref{eq:linearq} are 
\begin{equation}\label{eq:globalq5}
  \I_{\,L}\ \eqdef\ \;\int_{\,0}^{\,L}\bigl\{\,\eta\,u\ +\ b\,\eta_{\,x}\,u_{\,x} \,\bigr\}\;\ud x\,,
\end{equation}
\begin{equation}\label{eq:globalq6}
  {C}_{\,1\,L}\,(\eta,\,u)\ \eqdef\ \;\int_{\,0}^{\,L}\,\eta\;\ud x\,, \qquad 
  {C}_{\,2\,L}\,(\eta,\,u)\ \eqdef\ \;\int_{\,0}^{\,L}\,u\;\ud x\,.
\end{equation}

\begin{remark}\label{r1}
When the IVP of the MS System~\eqref{eq:5s}, \eqref{eq:6s} is considered on the space of smooth functions $(\eta,\,u)$ which vanish, along with their spatial derivatives, at infinity, then the preservation of the global energy and momentum
\begin{equation*}
  \Ec\,(\eta,\,u)\ \eqdef\ \int_{\,\R}\,\left(-\,\eta\,u\ +\ a\,\eta_{\,x}\,u_{\,x}\ -\ \S_{\,L}\,(\eta,\,u)\right)\,\ud x\,,
\end{equation*}
\begin{equation*}
  \Ic\,(\eta,\,u)\ \eqdef\ \frac{1}{2}\;\int_{\,\R}\,\left(\eta^{\,2}\ +\ u^{\,2}\ +\ b\,\eta_{\,x}^{\,2}\ +\ d\,u_{\,x}^{\,2}\right)\,\ud x
\end{equation*}
also holds. In this sense, the solitary wave solutions $\eta_{\,s}\ =\ \eta_{\,s}\,(X)\,$, $u_{\,s}\ =\ u_{\,s}\,(X)\,$, $X\ \eqdef\ x\ -\ c_{\,s}\,t$ of \eqref{eq:5s}, \eqref{eq:6s}, studied in \cite{Duran2019}, can be understood as MS relative equilibria, since the ODE system satisfied by the profiles $\eta_{\,s}\,$
, $u_{\,s}$ (see \cite[Equations~(3.2), (3.3)]{Duran2019}) can be written as
\begin{equation}\label{eq:msre}
  \delta\,\Ec\,(\eta_{\,s},\,u_{\,s})\ +\ c_{\,s}\,\delta\,\Ic\,(\eta_{\,s},\, u_{\,s})\ =\ 0\,,
\end{equation}
where $\delta\,(\scal)\ \eqdef\ \Bigl(\,\vdd{(-)}{\eta},\,\vdd{(-)}{u}\,\Bigr)^{\top}$ is the vector-valued variational derivative.
\end{remark}

\begin{remark}\label{r2}
When the IVP with zero boundary conditions at infinity or the periodic IVP of the $(a,\,b,\,a,\,b)$ System~\eqref{eq:5s}, \eqref{eq:6s} is multi-symplectic and \textsc{Hamiltonian}, we initially have four nonlinear conserved quantities. It is straightforward to see that in those cases
\begin{equation*}
  \delta\,\Ec\ =\ \begin{pmatrix} 0 & -\,1 \\ -\,1 & 0 \end{pmatrix} \delta\,\H\,, \qquad
  \delta\,\Ic\ =\ \begin{pmatrix} 0 & -\,1\\ -\,1 & 0\end{pmatrix} \delta\,\I\,,
\end{equation*}
leading to some functional dependence between $\Ec$ and $\H$ and between $\Ic$ and $\I\,$.
\end{remark}

%%% ----------------------------------------------------------------------- %%%

\section{Symplectic in space and symplectic in time methods}
\label{sec:sec3}

As mentioned in the introduction, the discretization of a MS System~\eqref{eq:ms} with symplectic one-step methods in space and time defines a multi-symplectic integrator in the sense of the preservation of a discrete version of the conservation law \eqref{eq:cons}, \cite{Bridges2001a}. Some properties and drawbacks of this approach have been described in several references, \cite{McLachlan2014}, and the purpose of this section is to analyze its application to the MS systems of the periodic problem for Equations~\eqref{eq:5s}, \eqref{eq:6s}. This will be done by studying the spatial discretization given by the Implicit Midpoint Rule (IMR) and the two-stage \textsc{Lobatto} IIIA--IIIB PRK method. On the one hand, the approximation with the IMR will serve us to illustrate the behaviour of the GLRK methods, typically used in the multi-symplectic integration, \cite{Reich2000a}. On the other hand, the two-stage \textsc{Lobatto} IIIA--IIIB method as spatial integrator will show if the alternative given by this family of PRK methods to generate an explicit semi-discretization, \cite{Ryland2008}, can improve the performance of the GLRK methods in our case.

The following notation will be used to discretize a general MS System~\eqref{eq:ms} with periodic boundary conditions on an interval $]\,0,\,L\,[\,$. For an integer $N\ \geq\ 1$ and a uniform grid $\{x_{\,j}\ =\ j\,h\ \vert\ j\ \in\ \Z\}$ with stepsize $h\ \eqdef\ \frac{L}{N}\,$, we consider the space $S_{\,h}$ of periodic vector functions $\Zz\ =\ (Z_{\,j})_{\,j\,\in\,\Z}$ defined on the grid, with $Z_{\,j}\ \in\ \R^{\,d}\,$, $Z_{\,j\,+\,N}\ =\ Z_{\,j}\,$, $j\ \in\ \Z\,$, while $S_{\,h}^{\,d\cdot N}$ will denote the space of vectors $(Z_{\,0},\,\ldots,\, Z_{\,N\,-\,1})^{\,\top}$ with $\Zz\ =\ (Z_{\,j})_{\,j\,\in\,\Z}\ \in\ S_{\,h}\,$. For the case of \eqref{eq:5s}, \eqref{eq:6s} in the MS form, we will make use of the previous definitions with $d\ =\ 10$ and \eqref{eq:const} -- \eqref{eq:pot}. We define the operators $\Dd_{x}\,$, $\Mm_{x}\,:\ S_{\,h}\ \longrightarrow\ S_{\,h}$ as
\begin{eqnarray}
  (\Dd_{\,x}\, Z)_{\,j}\ \eqdef\ \frac{Z_{\,j\,+\,1}\ -\ Z_{\,j}}{h}\,, \qquad 
  (\Mm_{\,x}\, Z)_{\,j}\ \eqdef\ \frac{Z_{\,j\,+\,1}\ +\ Z_{\,j}}{2}\,, \qquad j\ \in\ \Z\,.\label{eq:msop}
\end{eqnarray}

For $t\ \geq\ 0\,$, let $\z_{\,h}\,(t)\ =\ (\z_{\,h,\,j}\,(t))_{\,j\,\in\,\Z}\ \in\ S_{\,h}$ an approximation to the solution of the $L$-periodic IVP of \eqref{eq:ms} at $(x_{\,j},\,t)\,$, $j\ =\ 0,\,\ldots,\,N\,-\,1\,$. Because of periodicity, $\z_{\,h}\,(t)$ will be sometimes identified by its first $N$ components $\z_{\,h,\,j}\,(t)\,$, $j\ =\ 0,\,\ldots,\,N\,-\,1$ throughout the section. Similarly, for simplicity, $\Mm_{\,x}$ and $\Dd_{\,x}$ will also denote the restriction (with periodic boundary conditions) of the original operators to $S_{\,h}^{\,d\cdot N}$ and on $\R^{\,N}\,$, with matrix representations in this last case given respectively by
\begin{equation}\label{eq:dxmx}
  \widetilde{\Dd_{x}}\ \eqdef\ \frac{1}{h}\;\begin{pmatrix}
    -\,1 & 1 & 0 & \cdots & 0 \\
    0 & -\,1 & 1 & 0 & \cdots \\ 
    \vdots & \vdots & \ddots & \vdots & \vdots \\ 
    \vdots & \vdots & \vdots & \ddots & \vdots \\
    1 & 0 & 0 & \cdots & -\,1
  \end{pmatrix}\,, \qquad
  \widetilde{\Mm_{\,x}}\ \eqdef\ \frac{1}{2}\;\begin{pmatrix}
    1 & 1 & 0 & \cdots & 0 \\
    0 & 1 & 1 & 0 & \cdots \\
    \vdots & \vdots & \ddots & \vdots & \vdots \\ 
    \vdots & \vdots & \vdots & \ddots &\vdots \\
    1 & 0 & 0 & \cdots & 1
  \end{pmatrix}\,.
\end{equation}

%%% ----------------------------------------------------------------------- %%%

\subsection{Discretization with IMR}
\label{sec:glrk}

The spatial discretization of \eqref{eq:ms}, \eqref{eq:K} -- \eqref{eq:pot} with the IMR leads to the semi-discrete system
\begin{equation}\label{eq:msp1}
  \K\scal\od{}{t}\,\bigl(\Mm_{\,x}\scal\z_{\,h}\,(t)\bigr)_{\,j}\ +\ \M\scal\bigl(\Dd_{\,x}\scal\z_{\,h}\,(t)\bigr)_{\,j}\ =\ \grad_{\,\z}\,\S\,\bigl((\Mm_{\,x}\scal\z_{\,h}\,(t)\bigr))_{\,j}\,,
\end{equation}
for $j\ =\ 0,\,\ldots,\,N\,-\,1$ and where periodic boundary conditions are used. By using some properties of the \textsc{Kronecker} product of matrices, denoted by $\otimes\,$, in compact form \eqref{eq:msp1} reads
\begin{equation}\label{eq:msp2}
  \left(\Mm_{\,x}\otimes \K\right)\od{}{t}\,\z_{\,h}\,(t)\ +\ \left(\Dd_{\,x}\otimes\M\right)\z_{\,h}\,(t)\ =\ \grad\,\S\,\bigl(\Mm_{\,x}\scal\z_{\,h}\,(t)\bigr)\,,
\end{equation}
where on the left hand side of \eqref{eq:msp2} $\grad\,\S\,\bigl(\Mm_{\,x}\scal\z_{\,h}\,(t)\bigr)$ stands for the $10\cdot N$ vector with components $\grad_{\,\z}\,\S\,\bigl((\Mm_{\,x}\scal\z_{h}\,(t)\bigr)_{\,j}\ \in\ \R^{\,10}\,$, $j\ =\ 0,\,\ldots,\,N\,-\,1\,$. In terms of the components of $\z_{\,h,\,j}\ =\ (\eta_{\,j},\,\phi_{\,1,\,j},\,v_{\,1,\,j},\,w_{\,1,\,j},\,p_{\,1,\,j},\,u_{\,j},\,\phi_{\,2,\,j},\,v_{\,2,\,j},\,w_{\,2,\,j},\,p_{\,2,\,j})$ the discretization reads
\begin{eqnarray}
  \frac{1}{2}\;\od{}{t}\,\Mm_{\,x}\scal\phi_{\,1}\ -\ \frac{1}{2}\;b\,\od{}{t}\,\Mm_{\,x}\scal v_{\,1}\ +\ a\,\Dd_{\,x}\,v_{\,2}\ -\ \frac{1}{2}\;b\,\Dd_{\,x}\scal w_{\,1}\ &=&\ \Mm_{\,x}\scal p_{\,1}\ -\ \Mm_{\,x}\scal u\nonumber \\
  &&\ -\ \A\,(\Mm_{\,x}\scal\eta\,,\,\Mm_{\,x}\scal u)\,, \nonumber\\
  -\,\frac{1}{2}\;\od{}{t}\,\Mm_{\,x}\scal\eta\ -\ \Dd_{\,x}\scal p_{\,1}\ &=&\ 0\,, \nonumber\\
  \frac{1}{2}\;b\,\od{}{t}\,\Mm_{\,x}\scal\eta\ -\ a\,\Dd_{\,x}\scal u\ &=&\ \frac{1}{2}\;b\,\Mm_{\,x}\scal w_{\,1}\ -\ a\,\Mm_{\,x}\,v_{\,2}\,, \nonumber \\
  \frac{1}{2}\;b\,\Dd_{\,x}\scal\eta\ &=&\ \frac{1}{2}\;b\,\Mm_{\,x}\scal v_{\,1}\,, \label{eq:msp2a}\\
  \Dd_{\,x}\scal\phi_{\,1}\ &=&\ \Mm_{\,x}\scal\eta\,, \nonumber \\
  \frac{1}{2}\;\od{}{t}\,\Mm_{\,x}\scal\phi_{\,2}\ -\ \frac{1}{2}\;d\,\od{}{t}\,\Mm_{\,x}\scal v_{\,2}\ +\ c\,\Dd_{\,x}\scal v_{\,1}\ -\ \frac{1}{2}\;d\,\Dd_{\,x}\scal w_{\,2}\ &=&\ \Mm_{\,x}\scal p_{\,2}\ -\ \Mm_{\,x}\scal\eta \nonumber \\
  &&\ -\ \B\,(\Mm_{\,x}\scal\eta\,,\,\Mm_{\,x}\scal u)\,, \nonumber \\
  -\,\frac{1}{2}\;\od{}{t}\,\Mm_{\,x}\scal u\ -\ \Dd_{\,x}\scal p_{\,2}\ &=&\ 0\,, \nonumber \\
  \frac{1}{2}\;d\,\od{}{t}\Mm_{\,x}\scal u\ -\ a\,\Dd_{\,x}\scal\eta\ &=&\ \frac{1}{2}\;d\,\Mm_{\,x}\scal w_{\,2}\ -\ c\,\Mm_{\,x}\scal v_{\,1}\,, \nonumber \\
  \frac{1}{2}\;d\,\Dd_{\,x}\scal u\ &=&\ \frac{1}{2}\;d\,\Mm_{\,x}\scal v_{\,2}\,, \nonumber \\
  \Dd_{\,x}\scal\phi_{\,2}\ &=&\ \Mm_{\,x}\scal u\,.\nonumber
\end{eqnarray}
(Subindices were dropped for the sake of notation compactness.) A direct simplification of the variables $\phi_{\,j}\,$, $v_{\,j}\,$, $w_{\,j}\,$, $p_{\,j}\,$, $j\ =\ 1,\,2$ in \eqref{eq:msp2a} leads to
\begin{eqnarray}
  \left(\Mm_{\,x}^{\,3}\ -\ b\,\Dd_{\,x}^{\,2}\scal\Mm_{\,x}\right)\scal\od{}{t}\,\eta_{\,h}\ +\ a\,\Dd_{\,x}^{\,3}\scal u_{\,h}\ +\ \Mm_{\,x}^{\,2}\scal\Dd_{\,x}\scal u_{\,h} && \nonumber \\
  +\ \Dd_{\,x}\scal \Mm_{\,x}\scal\A\,(\Mm_{\,x}\scal\eta_{\,h}\,,\,\Mm_{\,x}\scal u_{\,h})\ &=&\ 0\,. \label{eq:msp3b} \\
  \left(\Mm_{\,x}^{\,3}\ -\ d\,\Dd_{\,x}^{\,2}\scal\Mm_{\,x}\right)\scal\od{}{t}\,u_{\,h}\ +\ a\,\Dd_{\,x}^{\,3}\scal\eta_{\,h}\ +\ \Mm_{\,x}^{\,2}\scal\Dd_{\,x}\scal\eta_{\,h} && \nonumber \\
  +\ \Dd_{\,x}\scal\Mm_{\,x}\scal\B\,(\Mm_{\,x}\scal\eta_{\,h}\,,\,\Mm_{\,x}\scal u_{\,h})\ &=&\ 0\,,\label{eq:msp4b}
\end{eqnarray}
where $\eta_{\,h}\ =\ (\eta_{\,h,\,0},\,\ldots,\,\eta_{\,h,\,N\,-\,1})^{\,\top}\,$, $u_{\,h}\ =\ (u_{\,h,\,0},\,\ldots,\,u_{\,h,\,N\,-\,1})^{\,\top}\,$. Note then that, in order for \eqref{eq:msp3b}, \eqref{eq:msp4b} to define an explicit ODE system, an odd number $N$ of nodes is required, since the operator $\Mm_{\,x}$ is invertible only in that case. This is how \cite[Theorem~2.1]{McLachlan2014}, concerning the spatial discretization of MS systems with GLRK methods, is illustrated in our case. The corresponding semi-discrete version of the MS conservation law \eqref{eq:cons} has the form
\begin{equation*}
  \od{}{t}\bigl((\Mm_{\,x}\scal \ud\z)_{\,j}\,\wedge\,(\Mm_{\,x}\otimes\K)\ud\z)_{\,j}\bigr)\ +\ \Dd_{\,x}\scal\bigl(\ud\z_{\,j}\,\wedge\,\M\scal\ud\z_{\,j}\bigr)\ =\ 0\,.
\end{equation*}

When $N$ is odd and the ODE System~\eqref{eq:msp3b}, \eqref{eq:msp4b} is integrated in time by a symplectic method, the result is a MS scheme in the sense that it preserves the corresponding fully discrete version of \eqref{eq:cons}, \cf \cite{Bridges2001}. Thus, for example, the \textsc{Preissman} Box scheme is derived when the time integrator is also the IMR,
\begin{equation}\label{eq:msp}
  \K\scal\Dd_{\,t}\scal\Mm_{\,x}\scal\z_{\,i}^{\,n}\ +\ \M\scal\Dd_{\,x}\scal\Mm_{\,t}\scal\z_{\,i}^{\,n}\ =\ \grad_{\,\z}\,\S\,\bigl(\Mm_{\,x}\scal\Mm_{\,t}\,\z_{\,i}^{\,n}\bigr)\,.
\end{equation}
where $\z^{\,n}\ \in\ S_{\,h}\,$, $n\ =\ 0,\,1,\,\ldots$ and
\begin{equation}\label{eq:msopb}
  \Dd_{\,t}\scal\z_{\,i}^{\,n}\ \eqdef\ \frac{\z_{\,i}^{\,n\,+\,1}\ -\ \z_{\,i}^{\,n}}{\Delta t}\,, \qquad 
  \Mm_{\,t}\scal\z_{\,i}^{\,n}\ \eqdef\ \frac{\z_{\,i}^{\,n\,+\,1}\ +\ \z_{\,i}^{\,n}}{2}\,,
\end{equation}
with $\Delta t$ as the time stepsize. A similar strategy to that in \cite{Ascher2004, Islas2005} for the KdV equation and NLS equation can be applied here to derive an equivalent version to \eqref{eq:msp}:
\begin{multline}
  \Dd_{\,t}\scal\Mm_{\,x}^{\,3}\scal\eta\ -\ b\,\Dd_{\,t}\scal\Dd_{\,x}^{\,2}\scal\Mm_{\,x}\scal\eta\ +\ a\,\Dd_{\,x}^{\,3}\scal\Mm_{\,t}\scal u\ +\ \Mm_{\,x}^{\,2}\scal\Mm_{\,t}\scal\Dd_{\,x}\scal u \\
  +\ \Dd_{\,x}\scal\Mm_{\,x}\scal\A\,(\Mm_{\,x}\scal\Mm_{\,t}\scal\eta\,,\,\Mm_{\,x}\scal\Mm_{\,t}\scal u)\ =\ 0\,. \label{eq:msp3}
\end{multline}
\begin{multline}
  \Dd_{\,t}\scal\Mm_{\,x}^{\,3}\scal u\ -\ d\,\Dd_{\,t}\scal\Dd_{\,x}^{\,2}\scal\Mm_{\,x}\scal u\ +\ a\,\Dd_{\,x}^{\,3}\scal\Mm_{\,t}\scal\eta\ +\ \Mm_{\,x}^{\,2}\scal\Mm_{\,t}\scal\Dd_{\,x}\scal\eta \\
  +\ \Dd_{\,x}\scal\Mm_{\,x}\scal\B\,(\Mm_{\,x}\scal\Mm_{\,t}\scal\eta\,,\,\Mm_{\,x}\scal\Mm_{\,t}\scal u)\ =\ 0\,, \label{eq:msp4}
\end{multline}
which shares the same discrete MS conservation law
\begin{equation}\label{eq:conl}
  \Dd_{\,t}\scal\left(\Mm_{\,x}\scal\ud\z_{\,i}^{\,n}\,\wedge\,\K\scal\Mm_{\,x}\scal\ud\z_{\,i}^{\,n}\right)\ +\ \Dd_{\,x}\scal\left(\Mm_{\,t}\scal\ud\z_{\,i}^{\,n}\,\wedge\,\M\scal\Mm_{\,t}\scal\ud\z_{\,i}^{\,n}\right)\ =\ 0\,.
\end{equation}
Additional properties of \eqref{eq:msp1} and \eqref{eq:msp} are described below.

%%% ----------------------------------------------------------------------- %%%

\subsubsection{Dispersion relation}

The first property is concerned with the dispersion relation of \eqref{eq:ms} or, equivalently, of the MS system
\begin{equation}\label{eq:ms2}
  \K\scal\z_{\,t}\ +\ \M\scal\z_{\,x}\ =\ \L\scal\z\,,
\end{equation}
where $\K$, $\M$ are given by \eqref{eq:K}, \eqref{eq:M} and
\begin{equation*}\label{eq:L}
  \L\ =\ \begin{pmatrix}[1.1]
    0 & 0 & 0 & 0 & 1 & -1 & 0 & 0 & 0 & 0 \\
    0 & 0 & 0 & 0 & 0 & 0 & 0 & 0 & 0 & 0 \\
    0 & 0 & 0 & \half b & 0 & 0 & 0 & -a & 0 & 0 \\
    0 & 0 & \half b & 0 & 0 & 0 & 0 & 0 & 0 & 0 \\
    1 & 0 & 0 & 0 & 0 & 0 & 0 & 0 & 0 & 0 \\
    -1 & 0 & 0 & 0 & 0 & 0 & 0 & 0 & 0 & 1 \\
    0 & 0 & 0 & 0 & 0 & 0 & 0 & 0 & 0 & 0 \\
    0 & 0 & -a & 0 & 0 & 0 & 0 & 0 & \half d & 0 \\
    0 & 0 & 0 & 0 & 0 & 0 & 0 & \half d & 0 & 0 \\
    0 & 0 & 0 & 0 & 0 & 1 & 0 & 0 & 0 & 0
  \end{pmatrix}\,,
\end{equation*}
is the (symmetric) matrix corresponding to the linear part of $\grad_{\,\z}\,\S\,(\z)$. The dispersion relation between frequency $\omega$ and wavenumber $k$ of \eqref{eq:ms2} reads
\begin{equation}\label{eq:dispr1}
  D\,(\omega,\,k)\ \eqdef\ {\rm det}(-\,\ui\,\omega\,\K\ +\ \ui\,k\,\M\ -\ \L)\ =\ 0\,,
\end{equation}
that is \cite{BCS}:
\begin{equation}\label{eq:dispr}
  \omega^{\,2}\,(k)\ =\ \frac{k^{\,2}\,(1\ +\ a\,k^{\,2})^{\,2}}{(1\ +\ b\,k^{\,2})(1\ +\ d\,k^{\,2})}\,.
\end{equation}
Now, those results in the literature concerning the numerical dispersion relation for GLRK methods can be applied to our case, see \cite{McLachlan2014, Frank2006}. Thus, for each $(\omega,\,k)$ satisfying \eqref{eq:dispr}, the IMR semi-discrete system \eqref{eq:msp1} has a discrete periodic solution
\begin{equation*}
  \z_{\,j}\,(t)\ =\ \ue^{\,\ui\,(\omega\,t\ +\ \xi\,j\,h)}\,a\,, \qquad a\ \in\ \R^{\,10},\, \qquad j\ =\ 0,\,\ldots,\,N\,-\,1\,,
\end{equation*}
with 
\begin{equation*}
  \frac{k\,h}{2}\ =\ \tan\frac{\xi\,h}{2}\,, \qquad \xi\,L\ =\ 2\,\pi\,p\,, \qquad \xi h\ \in\ ]\,l,\,l\,+\,2\,\pi\,p\,[\,,
\end{equation*}
for some $l\ \in\ \R\,$. Furthermore (see \cite[Corollary~2.4]{McLachlan2014}), the semi-discretization preserves the entire dispersion relation up to a diffeomorphic remapping of frequencies, for all $h\,$, with no parasitic waves, as well as the sign of the phase and group velocities. If the ODE system (with $N$ odd) is integrated by a GLRK method, then the corresponding numerical dispersion relation can be studied by using the results of \cite{Frank2006}. This is illustrated here with the Preissman scheme \eqref{eq:msp}. The preservation of \eqref{eq:dispr} is in the following sense, \cite{Ascher2004,Bridges2001, Bridges2006}: there are diffeomorphisms $\psi_{\,1}\,$, $\psi_{\,2}$ which conjugate the exact and the numerical dispersion relations such that to each pair $(\xi,\,\Omega)$ satisfying the numerical dispersion relation there corresponds a pair $(\psi_{\,1}\,(\xi),\,\psi_{\,2}\,(\Omega))$ satisfying the exact dispersion relation. By dropping the nonlinear terms in \eqref{eq:msp3}, \eqref{eq:msp4} both equations can be simplified to
\begin{eqnarray}
  \Dd_{\,t}^{\,2}\scal\Mm_{\,x}^{\,2}\bigl(\Mm_{\,x}^{\,2}\ -\ b\,\Dd_{\,x}^{\,2}\bigr)\scal\bigl(\Mm_{\,x}^{\,2} - d\,\Dd_{\,x}^{\,2}\bigr)\scal\eta\ -\ \Dd_{\,x}^{\,2}\scal\Mm_{\,t}^{\,2}\scal\bigl(\Mm_{\,x}^{\,2}\ +\ a\,\Dd_{\,x}^{\,2}\bigr)^{\,2}\scal\eta\ =\ 0\,. \label{eq:msp5}
\end{eqnarray}
And substituting $\eta_{\,j}^{\,n}\ =\ \ue^{\,\ui\,(j\,\xi\ +\ n\,\Omega)}$ the corresponding numerical dispersion relation is
\begin{equation*}
  \widetilde{D}\,(\xi,\,\Omega)\ \eqdef\ D\,(\psi_{\,1}\,(\xi),\,\psi_{\,2}\,(\Omega))\ =\ 0\,,
\end{equation*}
where $D$ is given by \eqref{eq:dispr1} and
\begin{equation*}
  \psi_{\,1}\,(\xi)\ \eqdef\ \frac{2}{h}\;\tan\,\Bigl(\frac{\xi}{2}\Bigr)\,, \qquad
  \psi_{\,2}\,(\xi)\ \eqdef\ \frac{2}{\Delta t}\;\tan\,\Bigl(\frac{\xi}{2}\Bigr)\,, \qquad \xi\ \in\ ]\,l,\,l\,+\,2\,\pi\,[\,.
\end{equation*}

%%% ----------------------------------------------------------------------- %%%

\subsubsection{Conservation properties}

The lack of skew-symmetry of the operator $\Dd_{\,x}$ determines the behaviour of \eqref{eq:msp1} and \eqref{eq:msp} with respect to discrete versions of the local energy and momentum, as well as the global quantities. By way of illustration, observe that natural discretizations of \eqref{eq:globalq1}, \eqref{eq:globalq2} are respectively $h\,\Ec_{\,h}$ and $h\,\Ic_{\,h}$ where
\begin{equation*}
  \Ec_{\,h}\,(t)\ \eqdef\ -\,\langle\eta_{\,h}\,(t),\,u_{\,h}\,(t)\rangle\ +\ a\,\langle\widetilde{\Dd_{\,x}}\scal\eta_{\,h}\,(t),\,\widetilde{\Dd_{\,x}}\scal u_{\,h}\,(t)\rangle\ -\ \langle \widetilde{\Gg}(\eta_{\,h}\,(t),\,u_{\,h}\,(t)),\, e_{\,h}\rangle\,,
\end{equation*}
\begin{multline*}
  \Ic_{\,h}\,(t)\ \eqdef\ \frac{1}{2}\;\bigl(\langle\eta_{\,h}\,(t),\, \eta_{\,h}\,(t)\rangle\ +\ \langle u_{\,h}\,(t),\, u_{\,h}\,(t)\rangle\\ 
  +\ b\,\langle\widetilde{\Dd_{\,x}}\scal\eta_{\,h}\,(t),\,\widetilde{\Dd_{\,x}}\scal\eta_{\,h}\,(t)\rangle\ +\ d\,\langle\widetilde{\Dd_{\,x}}\scal u_{\,h}\,(t),\, \widetilde{\Dd_{\,x}}\scal u_{\,h}\,(t)\rangle\bigr)\,,
\end{multline*}
where ${e}_{\,h}$ denotes the $N-$vector with all components equal to one and $\widetilde{\Gg}$ represents the corresponding version of \eqref{eq:globalq3} with the products understood in the \textsc{Hadamard} sense. We note then that the non skew-symmetry of $\widetilde{\Dd}_{\,x}$ in \eqref{eq:dxmx} prevents the preservation of $h\,\Ec_{\,h}$ and $h\,\Ic_{\,h}\,$. The same argument applies to the preservation of the natural discretization $h\,\I_{\,h}$ of the quadratic quantity $\I_{\,L}$ in \eqref{eq:globalq5} for the \textsc{Hamiltonian} case, where
\begin{equation*}
  \I_{\,h}\,(t)\ \eqdef\ \langle \eta_{\,h}\,(t),\, u_{\,h}\,(t)\rangle\ +\ b\,\langle \widetilde{\Dd_{\,x}}\scal\eta_{\,h}\,(t),\, \widetilde{\Dd_{\,x}}\scal u_{\,h}\,(t)\rangle\,.
\end{equation*}
Also because of the non skew-symmetry of $\Dd_{\,x}\,$, the \textsc{Hamiltonian} structure \eqref{eq:ham} is broken by the semi-discretization \eqref{eq:msp1}.

Finally, note that
\begin{equation}\label{eq:kernel}
  \widetilde{\Dd_{\,x}}^{\,\top}{e}_{\,h}\ =\ 0\,.
\end{equation}
Therefore, if we define
\begin{equation}\label{eq:lineard}
  C_{\,1,\,h}\,(t)\ =\ \langle \eta_{\,h}\,(t),\, {e}_{\,h}\rangle\,, \qquad
  C_{\,2,\,h}\,(t)\ =\ \langle u_{\,h}\,(t),\,{e}_{\,h}\rangle\,,
\end{equation}
then it is not hard to see that \eqref{eq:kernel} and the commutativity of $\Dd_{\,x}$ with the inverses of $\bigl(\Mm_{\,x}^{\,3}\ -\ b\,\Dd_{\,x}^{\,2}\scal\Mm_{\,x}\bigr)$ and $\bigl(\Mm_{\,x}^{\,3}\ -\ d\,\Dd_{\,x}^{\,2}\scal\Mm_{\,x}\bigr)$ imply
\begin{equation*}
  \od{}{t}\;C_{\,1,\,h}\,(t)\ =\ \od{}{t}\;C_{\,2,\,h}\,(t)\ =\ 0\,.
\end{equation*}
The preservation of the corresponding fully discrete versions of \eqref{eq:lineard} also holds when any symplectic time integrator (actually, any \textsc{Runge}--\textsc{Kutta} integrator) is applied, \cf \cite{Hairer2002}.

%%% ----------------------------------------------------------------------- %%%

\subsection{Discretization with PRK methods}
\label{sec:prk}

The use of symplectic PRK methods as an alternative to construct MS schemes was recently proposed in \cite{Ryland2008, McLachlan2014}. One of the reasons for that is in the possibility of avoiding the singular character of the ODE system after spatial discretization. This was analyzed in \cite{Ryland2008}, where conditions on the MS system for the family of \textsc{Lobatto} IIIA-IIIB PRK methods to lead to an explicit semi-discrete system were derived. These conditions (see \cite[Theorem~4.1]{Ryland2008}) are concerned with the structure of the MS system \eqref{eq:ms} and are not satisfied by corresponding one to the \textsc{Boussinesq} equations \eqref{eq:5s}, \eqref{eq:6s}. Specifically, the \textsc{Darboux} normal form of the matrix $\K$ is
\begin{equation*}\label{eq:Darboux}
   \begin{pmatrix}[1.1]
    0 & 0 & 0 & 0 & -1 & 0 & 0 & 0 & 0 & 0 \\
    0 & 0 & 0 & 0 & 0 & -1 & 0 & 0 & 0 & 0 \\
    0 & 0 & 0 & 0 & 0 & 0 & 0 & 0 & 0 & 0 \\
    0 & 0 & 0 & 0 & 0 & 0 & 0 & 0 & 0 & 0 \\
    1 & 0 & 0 & 0 & 0 & 0 & 0 & 0 & 0 & 0 \\
    0 & 1 & 0 & 0 & 0 & 0 & 0 & 0 & 0 & 0 \\
    0 & 0 & 0 & 0 & 0 & 0 & 0 & 0 & 0 & 0 \\
    0 & 0 & 0 & 0 & 0 & 0 & 0 & 0 & 0 & 0 \\
    0 & 0 & 0 & 0 & 0 & 0 & 0 & 0 & 0 & 0 \\
    0 & 0 & 0 & 0 & 0 & 0 & 0 & 0 & 0 & 0
  \end{pmatrix}\,.
\end{equation*}
However, the corresponding change of coordinates does not transform \eqref{eq:ms} into a system of the form required by the result in \cite{Ryland2008}. Thus the use of PRK methods to generate MS schemes in our case is an open question. A first approach was developed here, taking the two-stage \textsc{Lobatto} IIIA-IIIB method to discretize in space the system with several partitions. We observed that the form of the matrices \eqref{eq:K}, \eqref{eq:M} forces somehow to include the variables $v_{\,i}\,$, $w_{\,i}\,$, $i\ =\ 1,\,2$ into the same vector of any partition, and this complicates the derivation of a nonsingular ODE after spatial discretization. This can be illustrated by taking
\begin{equation*}
  \z^{\,(1)}\ =\ (\eta,\,\phi_{\,1},\,p_{\,1},\,u,\,\phi_{\,2},\,p_{\,2})\,, \qquad
  \z^{\,(2)}\ =\ (v_{\,1},\,w_{\,1},\,v_{\,2},\,w_{\,2})\,,
\end{equation*}
for which the discretization with the two-stage \textsc{Lobatto} IIIA-IIIB method leads to the system
\begin{eqnarray*}
  a\,\Dd_{\,x}\scal(v_{\,2})_{\,j\,-\,1/2}\ -\ \frac{1}{2}\;b\,\Dd_{\,x}\scal(w_{\,1})_{\,j\,-\,1/2} &=& (p_{\,1})_{\,j}\ -\ u_{\,j}\ -\ \frac{1}{2}\;\od{}{t}\;(\phi_{\,1})_{\,j} \nonumber \\
 && -\ \frac{1}{2}\;b\,\od{}{t}\;\Mm_{\,x}\scal(v_{\,1})_{\,j\,-\,1/2}\ -\ \A\,(\eta_{\,j}\,,\,u_{\,j})\,, \nonumber \\
  -\ \Dd_{\,x}\scal(p_{\,1})_{\,j} &=& \frac{1}{2}\;\od{}{t}\;\Mm_{\,x}\scal\eta_{\,j}\,, \nonumber \\
  -\ a\,\Dd_{\,x}\scal u_{\,j} &=& \frac{1}{2}\;b\,\Mm_{\,x}\scal(w_{\,1})_{\,j}\ -\ a\,\Mm_{\,x}\scal(v_{\,2})_{\,j}\ -\ \frac{1}{2}\;b\,\od{}{t}\;\Mm_{\,x}\scal\eta_{\,j}\,, \nonumber \\
  \frac{1}{2}\;b\,\Dd_{\,x}\scal\eta_{\,j} &=& \frac{1}{2}\;b\,\Mm_{\,x}\scal(v_{\,1})_{\,j}\,, \label{eq:msp2c} \\
  \Dd_{\,x}\scal(\phi_{\,1})_{\,j} &=& \Mm_{\,x}\scal\eta_{\,j}\,, \nonumber \\
  a\,\Dd_{\,x}\scal(v_{\,1})_{\,j\,-\,1/2}\ -\ \frac{1}{2}\;d\,\Dd_{\,x}\scal(w_{\,2})_{\,j\,-\,1/2} &=& (p_{\,2})_{\,j}\ -\ \eta_{\,j}\ -\ \frac{1}{2}\;\od{}{t}\;(\phi_{\,2})_{\,j} \nonumber \\
  && -\ \frac{1}{2}\;d\,\od{}{t}\;\Mm_{\,x}\scal(v_{\,d})_{\,j\,-\,1/2} -\ \B\,(\eta_{\,j}\,,\,u_{\,j})\,, \nonumber \\
   -\ \Dd_{\,x}\scal(p_{\,2})_{\,j} &=& \frac{1}{2}\;\od{}{t}\;\Mm_{\,x}\scal u_{\,j}\,, \nonumber \\
   -\ a\,\Dd_{\,x}\scal\eta_{\,j} &=& \frac{1}{2}\;d\,\Mm_{\,x}\scal(w_{\,2})_{\,j}\ -\ a\,\Mm_{\,x}\scal(v_{\,1})_{\,j}\ -\ \frac{1}{2}\;d\,\od{}{t}\;\Mm_{\,x}\scal u_{\,j}\,, \nonumber \\
  \frac{1}{2}\;d\,\Dd_{\,x}\scal u_{\,j} &=& \frac{1}{2}\;d\,\Mm_{\,x}\scal(v_{\,2})_{\,j}\,,\label{eq:msp2d}\\
  \Dd_{\,x}\scal(\phi_{\,2})_{\,j} &=& \Mm_{\,x}\scal u_{\,j}\,, \nonumber 
\end{eqnarray*}
which, after simplifications, has the form
\begin{equation}
  \bigl(\Mm_{\,x}\ -\ b\,\Dd_{\,x}\scal\Dd_{\,c}\bigr)\scal\od{}{t}\;\eta_{\,j}\ +\ a\,\Dd_{\,x}^{\,3}\scal u_{\,j\,-\,1}\ +\ \Dd_{\,x}\scal u_{\,j}\ +\ \Dd_{\,x}\scal\A\,(\eta_{\,j}\,,\,u_{\,j})\ =\ 0\,. \label{eq:msp3c}
\end{equation}
\begin{equation}
  \bigl(\Mm_{\,x}\ -\ d\,\Dd_{\,x}\scal\Dd_{\,c}\bigr)\scal\od{}{t}\;u_{\,j}\ +\ a\,\Dd_{\,x}^{\,3}\scal\eta_{\,j\,-\,1}\ +\ \Dd_{\,x}\scal\eta_{\,j}\ +\ \Dd_{\,x}\scal\B\,(\eta_{\,j}\,,\,u_{\,j})\ =\ 0\,. \label{eq:msp4c}
\end{equation}
where
\begin{equation*}
  (\Dd_{\,c}\scal Z)_{\,j}\ \eqdef\ \frac{Z_{\,j\,+\,1}\ -\ Z_{\,j\,-\,1}}{2\,h}\,,
\end{equation*}
and it can be seen that the matrices $\bigl(\Mm_{\,x}\ -\ b\,\Dd_{\,x}\scal\Dd_{\,c}\bigr)\,$, $\bigl(\Mm_{\,x}\ -\ d\,\Dd_{\,x}\scal\Dd_{\,c}\bigr)$ are nonsingular only when $N$ is odd, since $\Dd_{\,x}\scal\Dd_{\,c}\ =\ \Mm_{\,x}\scal\Dd_{\,x}^{\,2}\,$. This leads to a similar condition to that of the GLRK methods. No further approaches that may improve these results are however discarded.

\begin{remark}\label{r3}
One of these approaches may start from the natural partition
\begin{equation*}
  \z^{\,(1)}\ =\ (\eta,\,\phi_{\,1},\,v_{\,1},\,w_{\,1},\,p_{\,1})\,, \qquad
  \z^{\,(2)}\ =\ (u,\,\phi_{\,2},\,v_{\,2},\,w_{\,2},\,p_{\,2})\,,
\end{equation*}
which allows to write \eqref{eq:5s}, \eqref{eq:6s} in the form
\begin{equation}\label{eq:parti1}
  \begin{pmatrix}
    K\,(b) & 0 \\
    0 & K\,(d)
  \end{pmatrix}
  \partial_{\,t}\,\begin{pmatrix}
    \z^{\,(1)} \\
    \z^{\,(2)}
  \end{pmatrix}\ +\ 
  \begin{pmatrix}
    M\,(b) & M_{\,0} \\
    -\,M_{\,0} & M\,(d)
  \end{pmatrix}
  \partial_{\,x}\,
  \begin{pmatrix}
    \z^{\,(1)} \\
    \z^{\,(2)}
  \end{pmatrix}\ =\ 
  \begin{pmatrix}
    \grad_{\,\z^{\,(1)}}\,\S\,(\z^{\,(1)},\,z^{\,(2)}) \\
    \grad_{\,\z^{\,(2)}}\,\S\,(\z^{\,(1)},\,z^{\,(2)})
  \end{pmatrix}\,,
\end{equation}
where
\begin{equation*}
  K\,(\alpha)\ =\ \begin{pmatrix}
    0 & 1/2 & -\,\alpha/2 & 0 & 0 \\
    -\,1/2 & 0 & 0 & 0 & 0 \\
    \alpha/2 & 0 & 0 & 0 & 0 \\
    0 & 0 & 0 & 0 & 0 \\
    0 & 0 & 0 & 0 & 0 
  \end{pmatrix}\,, \quad
  M\,(\alpha)\ =\ \begin{pmatrix}
    0 & 0 & 0 & -\,\alpha/2 & 0 \\
    0 & 0 & 0 & 0 & -\,1 \\
    0 & 0 & 0 & 0 & 0 \\
    \alpha/2 & 0 & 0 & 0 & 0 \\
    0 & 1 & 0 & 0 & 0 
  \end{pmatrix}\,,
\end{equation*}
\begin{equation*}
  M_{\,0}\ =\ \begin{pmatrix}
    0 & 0 & a & 0 & 0 \\
    0 & 0 & 0 & 0 & 0 \\
    -\,a & 0 & 0 & 0 & 0 \\
    0 & 0 & 0 & 0 & 0 \\
    0 & 0 & 0 & 0 & 0 
  \end{pmatrix}\,.
\end{equation*}
The special form of \eqref{eq:parti1} may work out well for the application of another strategy with symplectic PRK methods.
\end{remark}

%%% ----------------------------------------------------------------------- %%%

\section{An alternative approach to MS discretizations}
\label{sec:sec4}

An alternative to construct geometric numerical methods for the periodic IVP of the MS System \eqref{eq:5s}, \eqref{eq:6s} is described in this section. As mentioned in the introduction, this is based on the search for approximate operators to the spatial partial derivative in order to obtain well-defined and geometric (in the numerical sense, \cite{Hairer2002}) semi-discretizations and their numerical integration with symplectic in time schemes. The approach generalizes that of previous references, \cite{Bridges2001a, Chen2001, Islas2006}, with spectral methods and some results can be extended to a general MS system \eqref{eq:ms}. They will be mentioned in the exposition below.

%%% ----------------------------------------------------------------------- %%%

\subsection{Semidiscretization in space}

Let $\Dd_{\,h}$ be some grid operator on $S_{\,h}^{\,N}$ approximating the first partial derivative in space. The semi-discretization, based on $\Dd_{\,h}\,$, which approximates the MS System \eqref{eq:ms} with periodic boundary conditions will have the form
\begin{equation}\label{eq:msp11}
  \K\scal\od{}{t}\;\bigl(\z_{\,h}\,(t)\bigr)_{\,j}\ +\ \M\scal\bigl(\Cc_{\,h}\scal\z_{\,h}\,(t)\bigr)_{\,j}\ =\ \grad_{\,\z}\,\S\,\bigl((\z_{\,h}\,(t))_{\,j}\bigr)\,,
\end{equation}
for $j\ =\ 0,\,\ldots,\,N\,-\,1\,$, where periodic boundary conditions are used and with $\Cc_{\,h}\ \eqdef\ \Dd_{\,h}\,\otimes\,\Id_{\,d}\,$. As in Section~\ref{sec:sec3}, the application of properties of the \textsc{Kronecker} product leads to the compact form of \eqref{eq:msp11}
\begin{equation}\label{eq:msp22}
 \bigl(\Id_{\,N}\,\otimes\,\K\bigr)\scal\od{}{t}\;\z_{\,h}\,(t)\ +\ \bigl(\Dd_{\,h}\,\otimes\,\M\bigr)\scal\z_{\,h}\,(t)\ =\ \grad\,\S\,\bigl(\Mm_{\,x}\scal\z_{\,h}\,(t)\bigr)\,,
\end{equation}
where on the left hand side of \eqref{eq:msp2} $\grad\,\S\,\bigl(\z_{\,h}\,(t)\bigr)$ stands for the $d\cdot N$ vector with components $\grad_{\,\z}\,\S\,\bigl((\z_{\,h}\,(t))_{\,j}\bigr)\ \in\ \R^{\,d}\,$, $j\ =\ 0,\,\ldots,\,N\,-\,1\,$.

The multi-symplectic character of the approximation \eqref{eq:msp22} is understood as follows \cite{Bridges2001}: if $U,\,V\ \in\ \R^{\,d\cdot N}$ are solutions of the variational equation associated to \eqref{eq:msp11}, then, using the skew-symmetry of $\M\,$, we have
\begin{equation}\label{eq:sdcl}
  \partial_{\,t}\,\bigl\langle\,\bigl(\Id_{\,N}\,\otimes\,\K\bigr)\scal U,\,V\,\bigr\rangle\ +\ \bigl\langle\,\bigl(\Dd_{\,h}\,\otimes\,\M\bigr)\scal U,\,V\,\bigr\rangle\ +\ \bigl\langle\,\bigl(\Dd_{\,h}^{\,\top}\,\otimes\,\M\bigr)\scal U,\,V\,\bigr\rangle\ =\ 0\,.
\end{equation}
Equation \eqref{eq:sdcl} is the MS conservation law preserved by \eqref{eq:msp22}. Note that, due to the structure of $\Id_{\,N}\,\otimes\,\K\,$, we have
\begin{equation*}
  \bigl\langle\,\bigl(\Id_{\,N}\,\otimes\,\K\bigr)\scal U,\,V\,\bigr\rangle\ =\ \sum_{j\,=\,0}^{N\,-\,1}\,\bigl\langle\,\K\scal U_{\,j},\,V_{\,j}\,\bigr\rangle\,,
\end{equation*}
where $U\ =\ (U_{\,j})_{\,j\,=\,0}^{\,N\,-\,1}\,$, $V\ =\ V_{\,j})_{\,j\,=\,0}^{\,N\,-\,1}\,$, $U_{\,j},\,V_{\,j}\ \in\ \R^{\,d}\,$. Consequently, if $\Dd_{\,h}$ is skew-symmetric, then \eqref{eq:sdcl} implies
\begin{equation*}
  \partial_{\,t}\,\sum_{j\,=\,0}^{N\,-\,1}\,\omega_{\,j}\ =\ 0\,, \qquad \omega_{\,j}\ \eqdef\ \bigl\langle\,\K\scal U_{\,j},\,V_{\,j}\,\bigr\rangle\,,
\end{equation*}
representing the semi-discrete version of the preservation of the total symplecticity $\bar{\omega}$ in \eqref{eq:tsymp}.

\begin{remark}\label{r4}
The requirement $\Dd_{\,h}^{\,\top}\ =\ -\,\Dd_{\,h}$ will be then the first property in this approach in order to construct MS discretizations. In this sense, two classical examples are introduced. The first one is the operator of central differences
\begin{equation}\label{eq:e1}
  (\Dd_{\,h}\scal V)_{\,j}\ \eqdef\ \frac{V_{\,j\,+\,1}\ -\ V_{\,j\,-\,1}}{2\,h}\,, \qquad j\ =\ 0,\,\ldots,\,N\,-\,1\,,
\end{equation} 
for $V\ =\ (V_{\,0},\,\ldots,\,V_{\,N\,-\,1})^{\,\top}\ \in\ S_{\,h}^{\,N}$ (the boundary conditions are used in \eqref{eq:e1} when defining the components $j\ =\ 0,\,N\,-\,1$) and matrix representation
\begin{equation}\label{eq:mre1}
  \widetilde{\Dd}_{\,h}= \frac{1}{2\,h}\;\begin{pmatrix}
    0 & 1 & 0 & \cdots & -\,1 \\
    -\,1 & 0 & 1 & 0 & \cdots \\
    \vdots & \vdots & \ddots & \vdots & \vdots \\
    \vdots & \vdots & \vdots & \ddots & \vdots \\
    1 & 0 & \cdots & -\,1 & 0
\end{pmatrix}.
\end{equation}
The second example of grid operator is given by the pseudospectral discretization. For each $Z\ \in\ S_{\,h}^{\,N}\,$, let $\widehat{Z}_{\,p}$ be the $p$\up{th}$-$discrete \textsc{Fourier} coefficient
\begin{eqnarray*}
  \widehat{Z}\,({p})=\frac{1}{N}\;\sum_{0\,\leq\,j\,\leq\,N}^{\prime\,\prime}\,Z_{\,j}\ue^{\,-\,\ui\,s\,p\,j\,h}\,, \qquad -\,\frac{N}{2}\ \leq\ p\ \leq\ \frac{N}{2}\,, \label{alm31}
\end{eqnarray*}
where $s\ \eqdef\ \frac{2\pi}{L}$ and the double prime in the sum denotes that the first and the last terms are divided by two. One can reconstruct $Z$ from the discrete \textsc{Fourier} coefficients by evaluating at the grid points the trigonometric interpolant polynomial
\begin{eqnarray}
  Z_{\,h}\,(x)\ =\ \sum_{-\,N/2\ \leq\ p\ \leq\ N/2}^{\prime\,\prime}\,\widehat{Z}\,({s\,p})\,\ue^{\,\ui\,p\,x}\,, \label{alm32}
\end{eqnarray}
in such a way that $Z_{\,j}\ =\ Z_{\,h}\,(x_{\,j})\,$. The pseudospectral differentiation operator on $Z$ is defined by differentiating \eqref{alm32} with respect to $x$ and evaluating at the $x_{\,j}\,$:
\begin{equation}\label{eq:alm33}
  (\Dd_{\,h}\scal Z)_{\,j}\ \eqdef\ \sum_{-\,N/2\, \leq\, p\, \leq\,N/2}^{\prime\,\prime}\,(\ui\,s\,p)\,\widehat{Z}\,(p)\ue^{\,\ui\,s\,p\,j\,h}\,, \qquad j\ \in\ \Z\,,
\end{equation}
which, in terms of the discrete \textsc{Fourier} coefficients, reads
\begin{eqnarray*}
  \widehat{(\Dd_{\,h}\,Z)}\,({p})\ =\ (\ui\,s\,p)\,\widehat{Z}\,({p})\,, \qquad -\,N/2\ \leq\ p\ \leq\ N/2\,. \label{alm34}
\end{eqnarray*}
\end{remark}

\begin{remark}\label{r5}
The general semi-discretization \eqref{eq:msp22}, along with the MS conservation law \eqref{eq:sdcl}, can be applied in particular to System~\eqref{eq:5s}, \eqref{eq:6s}, written in the MS form \eqref{eq:const} -- \eqref{eq:pot}. Note that after the direct simplification of the variables $\phi_{\,j}\,$, $v_{\,j}\,$, $w_{\,j}\,$, $p_{\,j}\,$, $j\ =\ 1,\,2$ the system \eqref{eq:msp11} can be written as
\begin{eqnarray}
  \bigl(\Id_{\,N}\ -\ b\,\Dd_{\,h}^{\,2}\bigr)\scal\od{}{t}\;\eta_{\,h}\ +\ \Dd_{\,h}\scal\bigl(\Id_{\,N}\ +\ a\,\Dd_{\,h}^{\,2}\bigr)\scal u_{\,h}\ +\ \Dd_{\,h}\scal\A\,(\eta_{\,h}\,,\,u_{\,h})\ &=&\ 0\,,\label{eq:msp33b} \\
  \bigl(\Id_{\,N}\ -\ d\,\Dd_{\,h}^{\,2}\bigr)\scal\od{}{t}\;u_{\,h}\ +\,\Dd_{\,h}\scal\bigl(\Id_{\,N}\ +\ a\,\Dd_{\,h}^{\,2}\bigr)\scal\eta_{\,h}\ +\ \Dd_{\,h}\scal\B\,(\eta_{\,h}\,,\,u_{\,h})\ &=&\ 0\,, \label{eq:msp44b}
\end{eqnarray}
where $\eta_{\,h}\ \eqdef\ (\eta_{\,h,\,0},\,\ldots,\,\eta_{\,h,\,N\,-\,1})^{\,\top}\,$, $u_{\,h}\ \eqdef\ (u_{\,h,\,0},\,\ldots,\,u_{\,h,\,N\,-\,1})^{\,\top}\,$. Thus, in order for \eqref{eq:msp33b}, \eqref{eq:msp44b} to define an explicit ODE system for the approximations $\eta_{\,h}\,$, $u_{\,h}\,$, the grid operator $\Dd_{\,h}$ must be chosen in such a way that
\begin{equation}\label{eq:nsing}
  \Nn_{\,h}\,(\alpha)\ \eqdef\ \Id_{\,N}\ -\ \alpha\,\Dd_{\,h}^{\,2}\,, \qquad \alpha\ >\ 0\,,
\end{equation}
is invertible in $S_{\,h}^{\,N}\,$. It is straightforward to check that \eqref{eq:nsing} is satisfied by the two examples of $\Dd_{\,h}$ described in Remark \ref{r4}.
\end{remark}

In order to provide additional geometric properties to the semi-discretization \eqref{eq:msp22}, some new conditions on $\Dd_{\,h}$ are required. This will be described in the following sections.

%%% ----------------------------------------------------------------------- %%%

\subsubsection{Dispersion relation}
\label{sec:412}

We assume that $\omega\,$, $k$ and $a\ \in\ \R^{\,d}$ satisfy the plane wave solution
\begin{equation}\label{eq:dispr1b}
  (-\,\ui\,\omega\,\K\ +\ \ui\,k\,\M\ -\ \L)\,a\ =\ 0\,,
\end{equation}
where $\L\ \eqdef\ \S^{\prime\prime}\,(\z)\,$. Now we look for solutions of the linearized system from \eqref{eq:ms} of the form
\begin{equation}\label{eq:soldisp1}
  \z_{\,h,\,j}\,(t)\ =\ \ue^{\,\ui\,(\omega\,t\ +\ \xi\,j\,h)}\,a\,.
\end{equation}
Note that in order to have $\z_{\,h}\,(t)\ \in\ S_{\,h}\,$, we need $\ue^{\,\ui\,N\,\xi\,h}\ \equiv\ 1\,$, that is
\begin{equation*}\label{soldisp2}
  \xi\ =\ \xi_{\,m}\ \equiv\ \frac{2\,\pi}{L}\;m\,, \qquad m\ \in\ \Z\,.
\end{equation*}
Susbstitution of \eqref{eq:soldisp1} into \eqref{eq:msp11} leads to
\begin{equation*}
  \bigl(\ui\,\omega\,\K\ +\ (\Dd_{\,h}\scal \ep)_{\,j}\,\overline{\ep}_{\,j}\scal\M\ -\ \L\bigr)\,a\ =\ 0\,,
\end{equation*}
where $\ep\ \eqdef\ \bigl(\,\ue^{\,\ui\,\xi\,h\,j}\,\bigr)_{\,j\,=\,0}^{\,N\,-\,1}\,$. Then, the dispersion relation yields
\begin{equation*}
  (\Dd_{\,h}\,\ep)_{\,j}\ =\ \ui\,k\,\ep_{\,j}\,, \qquad j\ =\ 0,\,\ldots,\,N\,-\,1\,,
\end{equation*}
that is
\begin{equation}\label{eq:soldisp3}
  (\Dd_{\,h}\ -\ \lambda\,\Id_{\,N})\scal\ep\ =\ 0, \qquad \lambda\ =\ \ui\,k\,.
\end{equation}
Note that if \eqref{eq:soldisp3} holds, then in particular
\begin{equation*}
  \ui\,k\ =\ (\Dd_{\,h}\,\ep)_{\,1}\overline{\ep}_{\,1}\,.
\end{equation*}
Thus, the identification of $\ep$ as eigenvector of $\Dd_{\,h}$ associated to the eigenvalue $\lambda\ =\ \ui\,k$ leads to some relation $k\ =\ k\,(\xi)\,$, which depends on the choice of $\Dd_{\,h}\,$. For the examples presented in Remark~\ref{r4}, one can check that $k\ =\ k\,(\xi)\ =\ \frac{\sin\,\xi\,h}{h}$ for \eqref{eq:e1} and $k\ =\ k\,(\xi)\ =\ \xi$ for \eqref{eq:alm33}.

%%% ----------------------------------------------------------------------- %%%

\subsubsection{Preservation of local conservation laws}

In this section we study the behaviour of the semi-discretization \eqref{eq:msp11} with respect to the local energy and momentum. We first make the following definitions. For $Q:\ S_{\,h}^{\,d\cdot N}\ \longrightarrow\ \R$ smooth
\begin{equation}\label{eq:dgrad1}
  \grad_{\,h}\,Q\,(\z)\ \eqdef\ Q^{\,\prime}\,(\z)\,\Cc_{\,h}\scal\z\ =\ \grad\,Q\,(\z)^{\,\top}\,\Cc_{\,h}\scal\z\,, \qquad \z\ \in\ S_{\,h}^{\,d\cdot N}\,,
\end{equation}
where $\Cc_{\,h}\ \eqdef\ \Dd_{\,h}\,\otimes\,\Id_{\,d}$ (For the case of \eqref{eq:5s}, \eqref{eq:6s}, $d\ =\ 10\,$.) The operator $\grad_{\,h}$ induces an operator (denoted in the same way for simplicity) on $S_{\,h}^{\,d}$ in the sense that if $\z\ =\ (\z_{\,j})_{\,j\,=\,0}^{\,N\,-\,1}\ \in\ S_{\,h}^{\,d\scal N}$ with $\z_{\,j}\ \in\ S_{\,h}^{\,d}$ and $Q\,:\ S_{\,h}^{\,d}\ \longrightarrow\ \R$ is smooth then we can define
\begin{equation}\label{eq:dgrad2}
  \grad_{\,h}\,Q\,(\z_{\,j})\ \eqdef\ Q^{\,\prime}\,(\z_{\,j})\,(\Cc_{\,h}\scal\z)_{\,j}\ =\ \grad\,Q\,(\z_{\,j})^{\,\top}\,(\Cc_{\,h}\scal\z)_{\,j}\,.
\end{equation}

On the other hand, if $\S\,:\ \R^{\,d}\ \longrightarrow\ \R\,$, we define $\widetilde{\S}_{\,h}\,:\ \R^{\,d\cdot N}\ \longrightarrow\ \R$ as
\begin{equation}\label{eq:dG}
  \widetilde{\S}_{\,h}\,(\z)\ =\ \sum_{j\,=\,0}^{N\,-\,1}\, \S\,(\z_{\,j})\,, \qquad \z\ =\ (\z_{\,j})_{j\,=\,0}^{N\,-\,1}\ \in\ S_{\,h}^{\,d\cdot N}\,.
\end{equation}
The operators \eqref{eq:dgrad1}, \eqref{eq:dgrad2} and \eqref{eq:dG} are necessary here and in the following subsection.

\begin{theorem}\label{th1}
Let $\z_{\,h}\,(t)\ =\ (\z_{\,h,\,j}\,(t))_{\,j\,=\,0}^{\,N\,-\,1}\ \in\ S_{\,h}^{\,d\cdot N}$ be a solution of \eqref{eq:msp22}. For $j\ =\ 0,\,\ldots,\,N\,-\,1$ define
\begin{eqnarray}
  \E_{\,j}\ &\eqdef&\ \S\,\bigl(\z_{\,h,\,j}\,(t)\bigr)\ -\ \frac{1}{2}\;(\z_{\,h,\,j}\,(t))^{\,\top}\,\M\scal\bigl(\Cc_{\,h}\scal\z_{\,h}\,(t)\bigr)_{\,j}\,, \nonumber \\
  \F_{\,j}\ &\eqdef&\ \frac{1}{2}\;\bigl(\z_{\,h,\,j}\,(t)\bigr)^{\,\top}\M\scal\od{\z_{\,h,\,j}}{t}\;(t)\,, \nonumber \\
  \I_{\,j}\ &\eqdef&\ \frac{1}{2}\;(\z_{\,h,\,j}\,(t))^{\,\top}\;\K\scal(\Cc_{\,h}\scal\z_{\,h}\,(t))_{\,j}\,, \nonumber \\
  \Gm_{\,j}\ &\eqdef&\ \S\,(\z_{\,h,\,j}\,(t))\ -\ \frac{1}{2}\;(\z_{\,h,\,j}\,(t))^{\,\top}\;\K\scal\od{\z_{\,h,\,j}}{t}\;(t)\,. \label{eq:sdlcl1}
\end{eqnarray}
Then 
\begin{equation}\label{eq:sdlcl2}
  \od{\E_{\,j}}{t}\ +\ \grad_{\,h}\,\F_{\,j}\ =\ 0\,, \quad  
  \od{\I_{\,j}}{t}\ +\ \grad_{\,h}\,\Gm_{\,j}\ =\ 0\,.
\end{equation}
where $\grad_{\,h}$ is given by \eqref{eq:dgrad2}.
\end{theorem}

\begin{proof}
Using the skew-symmetry of $\K$ and $\M$, \eqref{eq:msp11} and \eqref{eq:dgrad2}, we have
\begin{eqnarray*}
  \od{\E_{\,j}}{t}\ &=&\ \grad\,G\,(\z_{\,h,\,j})^{\,\top}\scal\od{\z_{\,h,\,j}}{t}\ -\ \frac{1}{2}\;\Bigl(\od{\z_{\,h,\,j}}{t}\Bigr)^{\,\top}\,\M\scal(\Cc_{\,h}\scal\z_{\,h})_{\,j} \\
  && -\ \frac{1}{2}\;(\z_{\,h,\,j})^{\,\top}\,\M\scal\Bigl(\Cc_{\,h}\scal\od{\z_{\,h}}{t}\Bigr)_{\,j} \\
  &=& \grad\,G\,(\z_{\,h,\,j})^{\,\top}\scal\od{\z_{\,h,\,j}}{t}\ -\ \frac{1}{2}\;\Bigl(\od{\z_{\,h,\,j}}{t}\Bigr)^{\,\top}\scal\Bigl(\grad\,G\,(\z_{\,h,\,j})\ -\ \K\scal\od{\z_{\,h,\,j}}{t}\Bigr) \\
  && -\ \frac{1}{2}\;(\z_{\,h,\,j})^{\,\top}\;\M\scal(\Cc_{\,h}\scal\od{\z_{\,h}}{t})_{\,j} \\
  &=& \frac{1}{2}\;\Bigl(\od{\z_{\,h,\,j}}{t}\Bigr)^{\,\top}\;\grad\,G\,(\z_{\,h,\,j})\ -\ \frac{1}{2}\;\bigl(\Cc_{\,h}\scal\od{\z_{\,h}}{t}\bigr)_{\,j}\;\M^{\,\top}\scal(\z_{\,h,\,j}) \\
  &=& \frac{1}{2}\;\Bigl(\od{\z_{\,h,\,j}}{t}\Bigr)^{\,\top}\bigl(\K\scal\od{\z_{\,h,\,j}}{t}\ +\ \M\scal\bigl(\Cc_{\,h}\scal\z_{\,h}\,(t)\bigr)_{\,j}\bigr) \\
  && -\ \frac{1}{2}\;\Bigl(\Cc_{\,h}\scal\od{\z_{\,h}}{t}\Bigr)_{\,j}\;\M^{\,\top}\scal(\z_{\,h,\,j}) \\
  &=&\ \frac{1}{2}\;\Bigl(\od{\z_{\,h,\,j}}{t}\Bigr)^{\,\top}\;\M\scal\bigl(\Cc_{\,h}\scal\z_{\,h}\,(t)\bigr)_{\,j}\ +\ \frac{1}{2}\;\Bigl(\Cc_{\,h}\scal\od{\z_{\,h}}{t}\Bigr)_{\,j}\;\M\scal(\z_{\,h,\,j})\\
  &=& -\ \grad_{\,h}\,\F_{\,j}\,.
\end{eqnarray*}
Similarly,
\begin{eqnarray*}
  \grad_{\,h}\,\Gm_{\,j}\ &=&\ \grad\,G\,(\z_{\,h,\,j})^{\,\top}\scal(\Cc_{\,h}\scal\z_{\,h})_{\,j}\ -\ \frac{1}{2}\;\Bigl(\Cc_{\,h}\scal\z_{\,h}\Bigr)_{\,j}^{\,\top}\scal\K\scal(\od{\z_{\,h,\,j}}{t}) \\
  && -\ \frac{1}{2}\;(\z_{\,h,\,j})^{\,\top}\scal\K\scal\Bigl(\Cc_{\,h}\scal\od{\z_{\,h}}{t}\Bigr)_{\,j} \\
  &=& (\Cc_{\,h}\scal\z_{\,h})_{\,j}^{\,\top}\,\Bigl(\K\scal\od{\z_{\,h,\,j}}{t}\ +\ \M\scal(\Cc_{\,h}\scal\z_{\,h}\,(t))_{\,j}\Bigr)\ -\ \frac{1}{2}\;(\Cc_{\,h}\scal\z_{\,h})_{\,j}^{\,\top}\scal\K\scal\od{\z_{\,h,\,j}}{t} \\
  && -\ \frac{1}{2}\;(\z_{\,h,\,j})^{\,\top}\;\K\scal\Bigl(\Cc_{\,h}\scal\od{\z_{\,h}}{t}\Bigr)_{\,j} \\
  &=& \frac{1}{2}\;(\Cc_{\,h}\scal\z_{\,h})_{\,j}^{\,\top}\;\K\scal\od{\z_{\,h,\,j}}{t}\ -\ \frac{1}{2}\;(\z_{\,h,\,j})^{\,\top}\,\K\scal\Bigl(\Cc_{\,h}\scal\od{\z_{\,h}}{t}\Bigr)_{\,j} \\
  &=& -\ \frac{1}{2}\;\Bigl(\od{\z_{\,h,\,j}}{t}\Bigr)^{\,\top}\,\K\scal(\Cc_{\,h}\scal\z_{\,h})_{\,j}\ -\ \frac{1}{2}\;(\z_{\,h,\,j})^{\,\top}\,\K\scal\Bigl(\Cc_{\,h}\scal\od{\z_{\,h}}{t}\Bigr)_{\,j}\\ 
  &=& -\ \od{\I_{\,j}}{t}\,.
\end{eqnarray*}
\end{proof}

\begin{remark}\label{r7}
Theorem \ref{th1} applies to \eqref{eq:5s}, \eqref{eq:6s} with $d\ =\ 10$ and the elements of the MS formulation given by \eqref{eq:K} -- \eqref{eq:pot}.
\end{remark}

%%% ----------------------------------------------------------------------- %%%

\subsubsection{Preservation of global quantities and Hamiltonian case}

As mentioned in Section \ref{sec:sec2}, the presence of periodic boundary conditions implies the preservation of the global energy and momentum given by \eqref{eq:globalq1} and \eqref{eq:globalq2} respectively. The semi-discrete version, based on \eqref{eq:msp11}, of this result is as follows.

\begin{theorem}\label{th2}
Let us assume that $\Dd_{\,h}$ is skew-symmetric and $\z_{\,h}\,(t)\ =\ \bigl(\z_{\,h,\,j}\,(t)\bigr)_{\,j\,=\,0}^{\,N\,-\,1}\ \in\ S_{\,h}^{\,d\cdot N}$ be a solution of \eqref{eq:msp22} and define
\begin{equation}\label{eq:sdenergy}
  \Ec_{\,h}\ \eqdef\ \widetilde{\S}_{\,h}\,(\z_{\,h}\,(t))\ -\ \frac{1}{2}\;\z_{\,h}\,(t)^{\,\top}\,(\Dd_{\,h}\,\otimes\,\M)\z_{\,h}\,(t)\,,
\end{equation}
where $\widetilde{\S}_{\,h}$ is given by \eqref{eq:dG}. Then $\od{\Ec_{\,h}}{t}\ =\ 0\,$.
\end{theorem}

\begin{proof}
Using the skew-symmetry of $\K\,$, $\M\,$, $\Dd_{\,h}$ and the semi-discretization \eqref{eq:msp22}, the following holds:
\begin{eqnarray*}
  \od{\Ec_{\,h}}{t}\ &=&\ \sum_{j\,=\,0}^{N\,-\,1}\,\grad\,\S\,(\z_{\,h,\,j})\;\od{\z_{\,h,\,j}}{t}\ -\ \frac{1}{2}\;\left(\od{\z_{\,h}}{t}\right)^{\,\top}\,(\Dd_{\,h}\,\otimes\,\M)\scal\z_{\,h} \\
  && -\ \frac{1}{2}\;\z_{\,h}^{\,\top}\,(\Dd_{\,h}\,\otimes\,\M)\scal\od{\z_{\,h}}{t} \\
  &=&\ \grad\,\widetilde{\S}_{\,h}\,(\z_{\,h})^{\,\top}\scal\od{\z_{\,h}}{t}\ -\ \frac{1}{2}\;\left(\od{\z_{\,h}}{t}\right)^{\,\top}\left(\grad\,\widetilde{\S}_{\,h}\,(\z_{\,h})\ -\ (\Id_{\,N}\,\otimes\,\K)\scal\od{\z_{\,h}}{t}\right) \\
  && -\ \frac{1}{2}\;\z_{\,h}^{\,\top}\,(\Dd_{\,h}\,\otimes\,\M)\scal\od{\z_{\,h}}{t} \\
  &=& \frac{1}{2}\;\grad\,\widetilde{\S}_{\,h}\,(\z_{\,h})^{\,\top}\scal\od{\z_{\,h}}{t}\ -\ \frac{1}{2}\;\z_{\,h}^{\,\top}\,(\Dd_{\,h}\,\otimes\,\M)\scal\od{\z_{\,h}}{t} \\
  &=& \frac{1}{2}\;\grad\,\widetilde{\S}_{\,h}\,(\z_{\,h})^{\,\top}\scal\od{\z_{\,h}}{t}\ -\ \frac{1}{2}\;\left((\Dd_{\,h}^{\,\top}\,\otimes\,\M^{\,\top})\scal\z_{\,h}\right)^{\,\top}\;\od{\z_{\,h}}{t} \\
  &=& \frac{1}{2}\;\left(\grad\,\widetilde{\S}_{\,h}(\z_{\,h})\ -\ (\Dd_{\,h}\,\otimes\,\M)\scal\z_{\,h}\right)^{\,\top}\;\od{\z_{\,h}}{t} \\
  &=& \frac{1}{2}\;\left((\Id_{\,N}\,\otimes\,\K)\scal\od{\z_{\,h}}{t}\right)^{\,\top}\;\od{\z_{\,h}}{t}\ =\ 0\,.
\end{eqnarray*}
\end{proof}

\begin{remark}\label{r8}
Note that $h\,\Ec_{\,h}$ is a natural discretization of \eqref{eq:globalq1}. In the case of \eqref{eq:5s}, \eqref{eq:6s} with $d\ =\ 10$ and the elements of the MS formulation given by \eqref{eq:K} -- \eqref{eq:pot}, if we assume \eqref{eq:nsing}, then the discrete energy can be written as
\begin{equation*}
  \Ec_{\,h}\,(t)\ \eqdef\ -\ \bigl\langle\,\eta_{\,h}\,(t)\,,\,u_{\,h}\,(t)\,\rangle\rangle\ +\ a\,\bigl\langle\,\widetilde{\Dd}\,\eta_{\,h}\,(t)\,,\,\widetilde{\Dd}_{\,h}\,u_{\,h}\,(t)\,\bigr\rangle\ -\ \bigl\langle\,\widetilde{\Gg}\,(\eta_{\,h}\,(t)\,,\,u_{\,h}\,(t))\,,\,e_{\,h}\bigr\rangle\,,
\end{equation*}
where, as before, ${e}_{\,h}$ denotes the $N-$vector with all components equal to one, $\widetilde{\Dd}_{\,h}$ is given by \eqref{eq:mre1} and $\widetilde{\Gg}$ represents the corresponding version of \eqref{eq:globalq3} with the products understood in the \textsc{Hadamard} sense.
\end{remark}

As far as the global momentum, the behaviour of the semi-discrete version of \eqref{eq:globalq2} is determined by the following result.

\begin{theorem}\label{th3}
Let us assume that $\Dd_{\,h}$ is skew-symmetric, let $\z_{\,h}\,(t)\ =\ \bigl(\z_{\,h,\,j}\,(t)\bigr)_{\,j\,=\,0}^{\,N\,-\,1}\ \in\ S_{\,h}^{\,d\cdot N}$ be a solution of \eqref{eq:msp22} and define
\begin{eqnarray}
  \Ic_{\,h}\ &\eqdef&\ \frac{1}{2}\;\z_{\,h}\,(t)^{\,\top}\,(\Dd_{\,h}\,\otimes\,\K)\scal\z_{\,h}\,(t)\,, \label{eq:sdmomentum2} \\
  \Gm_{\,h}\ &\eqdef&\ \widetilde{\S}_{\,h}\,\bigl(\z_{\,h}\,(t)\bigr)\ -\ \frac{1}{2}\;\z_{\,h}\,(t)^{\,\top}\,(\Id_{\,N}\,\otimes\,\K)\scal\od{\z_{\,h}}{t}\,(t)\,, \nonumber
\end{eqnarray}
where $\widetilde{\S}_{\,h}$ is given by \eqref{eq:dG}. Then,
\begin{equation}\label{eq:sdgcl}
  \od{\Ic_{\,h}}{t}\ +\ \grad_{\,h}\,\Gm_{\,h}\ =\ 0\,,
\end{equation} 
where $\grad_{\,h}$ is given by \eqref{eq:dgrad1}.
\end{theorem}

\begin{proof}
By using similar arguments to those in Theorem~\ref{th2} and additional properties of the \textsc{Kronecker} product, we can write
\begin{eqnarray}
  \od{\Ic_{\,h}}{t}\ &=&\ \z_{\,h}^{\,\top}\,(\Dd_{\,h}\,\otimes\,\K)\scal\od{\z_{\,h}}{t}\ =\ \z_{\,h}^{\,\top}\,(\Dd_{\,h}\,\otimes\,\Id_{\,d})\scal(\Id_{\,N}\,\otimes\,\K)\scal\od{\z_{\,h}}{t} \nonumber \\
  &=&\ \z_{\,h}^{\,\top}\,(\Dd_{\,h}\,\otimes\,\Id_{\,d})\scal\bigl(\grad\,\widetilde{\S}_{\,h}(\z_{\,h})\ -\ (\Dd_{\,h}\,\otimes\,\M)\scal\z_{\,h}\bigr) \nonumber \\
  &=&\ \z_{\,h}^{\,\top}(\Dd_{\,h}\,\otimes\,\Id_{\,d})\scal\grad\,\widetilde{\S}_{\,h}(\z_{\,h})\ -\ \z_{\,h}^{\,\top}\,(\Dd_{\,h}\,\otimes\,\Id_{\,d})\scal(\Dd_{\,h}\,\otimes\,\M)\scal\z_{\,h}\,. \label{eq:aux1}
\end{eqnarray}
Notice that
\begin{equation*}
  \z_{\,h}^{\,\top}\,(\Dd_{\,h}\,\otimes\,\Id_{\,d})\scal(\Dd_{\,h}\,\otimes\,\M)\scal\z_{\,h}\ =\ \z_{\,h}^{\,\top}\,(\Dd_{\,h}^{\,2}\,\otimes\,\M)\scal\z_{\,h}\,.
\end{equation*}
Therefore, since $\Dd_{\,h}^{\,2}\,\otimes\,\M$ is skew-symmetric, the second term in \eqref{eq:aux1} vanishes. On the other hand, using \eqref{eq:dgrad1} we have
\begin{eqnarray*}
  \grad_{\,h}\,\Gm_{\,h}\ &=&\ \grad\,\widetilde{\S}_{\,h}\,(\z_{\,h})^{\,\top}(\Cc_{\,h}\scal\z_{\,h})\ -\ \frac{1}{2}\;(\Cc_{\,h}\scal\z_{\,h})^{\,\top}\,(\Id_{\,N}\,\otimes\,\K)\scal\od{\z_{\,h}}{t} \\
  && -\ \frac{1}{2}\;\z_{\,h}^{\,\top}\,(\Id_{\,N}\,\otimes\,\K)\scal\Bigl(\Cc_{\,h}\scal\od{\z_{\,h}}{t}\Bigr) \\
  &=&\ \grad\,\widetilde{\S}_{\,h}\,(\z_{\,h})^{\,\top}\,(\Cc_{\,h}\scal\z_{\,h})\ -\ \frac{1}{2}\;(\z_{\,h})^{\,\top}\,(\Dd_{\,h}^{\,\top}\,\otimes\,\Id_{\,d})\scal(\Id_{\,N}\,\otimes\,\K)\scal\od{\z_{\,h}}{t} \\
  && -\ \frac{1}{2}\;\z_{\,h}^{\,\top}\,(\Id_{\,N}\,\otimes\,\K)\scal(\Dd_{\,h}\,\otimes\,\Id_{\,d})\scal\od{\z_{\,h}}{t} \\
  &=&\ \grad\,\widetilde{\S}_{\,h}(\z_{\,h})^{\,\top}\,(\Dd_{\,h}\,\otimes\,\Id_{\,d})\scal\z_{\,h}\ -\ \frac{1}{2}\;(\z_{\,h})^{\,\top}\,(\Dd_{\,h}^{\,\top}\,\otimes\,\K)\scal\od{\z_{\,h}}{t} \\
  && -\ \frac{1}{2}\;\z_{\,h}^{\,\top}\,(\Dd_{\,h}\,\otimes\,\K)\;\od{\z_{\,h}}{t} \\
  &=& \grad\,\widetilde{\S}_{\,h}\,(\z_{\,h})^{\,\top}\,(\Dd_{\,h}\,\otimes\,\Id_{\,d})\scal\z_{\,h}\ =\ -\,\z_{\,h}^{\,\top}\,(\Dd_{\,h}^{\,\top}\,\otimes\,\Id_{\,d})\scal\grad\,\widetilde{\S}_{\,h}(\z_{\,h})\,,
\end{eqnarray*}
which, along with \eqref{eq:aux1}, leads to \eqref{eq:sdgcl}.
\end{proof}

\begin{remark}\label{r9}
In the case of \eqref{eq:5s}, \eqref{eq:6s} with $d\ =\ 10$ and the elements of the MS formulation given by \eqref{eq:K} -- \eqref{eq:pot},  if we assume \eqref{eq:nsing}, then the semi-discrete global momentum \eqref{eq:sdmomentum2} can be written as
\begin{multline*}
  \Ic_{\,h}\,(t)\ \eqdef\ \frac{1}{2}\;\Bigl(\,\bigl\langle\,\eta_{\,h}\,(t)\,,\,\eta_{\,h}\,(t)\,\bigr\rangle\ +\ \bigl\langle\,u_{\,h}\,(t)\,,\,u_{\,h}\,(t)\,\bigr\rangle\ +\\ 
  b\,\bigl\langle\,\Dd_{\,h}\scal\eta_{\,h}\,(t)\,,\,\Dd_{\,h}\scal\eta_{\,h}\,(t)\,\bigr\rangle\ +\ d\,\bigl\langle\,\Dd_{\,h}\scal u_{\,h}\,(t)\,,\,\Dd_{\,h}\scal u_{\,h}\,(t)\,\bigr\rangle\,\Bigr)\,.
\end{multline*}
Note that, using the skew-symmetry of $\Dd_{\,h}$ and the form of $\A$ and $\B$ in \eqref{eq:5s}, \eqref{eq:6s}, straightforward but tedious computations show that
\begin{eqnarray}\label{eq:order1}
  \od{\Ic_{\,h}}{t}\ =\ \bigl\langle\,\A_{\,h}\,(\eta_{\,h}\,,\,u_{\,h})\,,\,\Dd_{\,h}\scal\eta_{\,h}\,\bigr\rangle\ +\ \bigl\langle\,\B_{\,h}\,(\eta_{\,h}\,,\,u_{\,h})\,,\, \Dd_{\,h}\scal u_{\,h}\bigr\rangle\,,
\end{eqnarray}
where $\A_{\,h}\,$, $\B_{\,h}$ are corresponding versions of $\A\,$, $\B$ with \textsc{Hadamard} products. Then, if $s_{\,1}$ and $s_{\,2}$ are, respectively, the order of approximation of $\eta_{\,h}\,$, $u_{\,h}$ to $\eta\,$, $u$ at the grid and of $\Dd_{\,h}$ to the spatial partial derivative, then
\begin{eqnarray*}
  \od{\Ic_{\,h}}{t}\ =\ \O\,(h^{\,r})\,, \qquad r\ \eqdef\ \min\{s_{\,1}\,,\,s_{\,2}\}\,.
\end{eqnarray*}
The preservation of \eqref{eq:sdmomentum2} holds when some symmetry conditions are additionally imposed, in the same line to that considered in \cite{Cano2006} for the discretization of the nonlinear wave equation and the NLS equation. More specifically, we define the $N\times N$ matrix $\D$ which reverses the order of the components of the vector to which it is applied (\ie $(\D\scal x)_{\,j}\ =\ x_{\,N\,-\,j\,-\,1}\,$, $j\ =\ 0,\,\ldots,\,N\,-\,1$),
\begin{equation*}
  \D\ \eqdef\ \begin{pmatrix}
    0 & 0 & 0 & \cdots & 1 \\
    0 & 0 & \cdots & 1 & 0 \\
    \vdots & \vdots & \ddots & \vdots & \vdots \\
    \vdots & \vdots & \vdots & \ddots & \vdots \\
    1 & 0 & 0 & \cdots & 0
  \end{pmatrix}\,.
\end{equation*}
Assume now that
\begin{equation}\label{eq:mconserv1}
  \D\scal\eta_{\,h}\ =\ \eta_{\,h}\,, \qquad  \D\scal u_{\,h}\ =\ u_{\,h}\,,
\end{equation}
\begin{equation}\label{eq:mconserv2}
  \Dd_{\,h}\scal\D\ =\ -\,\D\scal\Dd_{\,h}\,.
\end{equation}
Then one can check that
\begin{equation*}
  \D\scal\A_{\,h}\,(\eta_{\,h}\,,\,u_{\,h})\ =\ \A_{\,h}\,(\eta_{\,h}\,,\,u_{\,h})\,, \qquad
  \D\scal\B_{\,h}\,(\eta_{\,h}\,,\,u_{\,h})\ =\ \B_{\,h}\,(\eta_{\,h}\,,\,u_{\,h})\,.
\end{equation*}
Therefore (\cf \cite{Cano2006})
\begin{eqnarray*}
  \bigl\langle\,\A_{\,h}\,(\eta_{\,h}\,,\,u_{\,h})\,,\,\Dd_{\,h}\scal\eta_{\,h}\,\bigr\rangle\ &=&\ \bigl\langle\,\D\scal\A_{\,h}\,(\eta_{\,h}\,,\,u_{\,h})\,,\,\D\scal\Dd_{\,h}\scal\eta_{\,h}\,\bigr\rangle \\
  &=&\ -\,\bigl\langle\,\D\scal\A_{\,h}\,(\eta_{\,h}\,,\,u_{\,h})\,,\,\Dd_{\,h}\scal\D\scal\eta_{\,h}\,\bigr\rangle\ =\\ 
  && -\,\bigl\langle\,\A_{\,h}\,(\eta_{\,h}\,,\,u_{\,h})\,,\, \Dd_{\,h}\scal\eta_{\,h}\bigr\rangle\,,
\end{eqnarray*}
which implies $\bigl\langle\,\A_{\,h}\,(\eta_{\,h}\,,\,u_{\,h})\,,\,\Dd_{\,h}\scal\eta_{\,h}\,\bigr\rangle\ =\ 0\,$. Similarly we have $\bigl\langle\,\B_{\,h}\,(\eta_{\,h}\,,\,u_{\,h})\,$, $\Dd_{\,h}\scal u_{\,h}\,\bigr\rangle\ =\ 0$ and, according to \eqref{eq:order1}, the preservation of \eqref{eq:sdmomentum2} holds. This result can be applied, for example, to the approximations $\eta_{\,h}\,$, $u_{\,h}$ of (classical or generalized) solitary-wave solutions, \cite{Duran2019}, satisfying \eqref{eq:mconserv1}, see the experiments in Section~\ref{sec:422}. Observe finally that condition \eqref{eq:mconserv2} is satisfied by the two examples of $\Dd_{\,h}$ described in Remark~\ref{r4}.
\end{remark}

\begin{remark}\label{r10}
A similar study can be made in the \textsc{Hamiltonian} case. Thus, if $\Dd_{\,h}$ is skew-symmetric, \eqref{eq:const3b} and \eqref{eq:nsing} hold, then it is straightforward to see that the $(a,\,b,\,a,\,b)$ semi-discretization \eqref{eq:msp33b}, \eqref{eq:msp44b} preserves a \textsc{Hamiltonian} formulation with respect to the structure given by the $2N\,\times\, 2N$ matrix
\begin{equation*}
  J_{\,h}\ \eqdef\ -\,\begin{pmatrix}
    \mathbf{0}_{\,N} & \Dd_{\,h}\scal(\Id_{\,N}\ -\ b\,\Dd_{\,h}^{\,2})^{\,-\,1} \\
    \Dd_{\,h}\scal(\Id_{\,N}\ -\ b\,\Dd_{\,h}^{\,2})^{\,-\,1} & \mathbf{0}_{\,N}
    \end{pmatrix}\,,
\end{equation*}
where $\mathbf{0}_{\,N}$ denotes the $N\times N$ zeros matrix. The discrete \textsc{Hamiltonian} is given by 
\begin{multline}\label{eq:dHam}
  \H_{\,h}\,(U,\,V)\ \eqdef\ \frac{1}{2}\;\Bigl(\bigl\langle\,U\,,\,U\,\bigr\rangle\ +\ \bigl\langle\,V\,,\,V\,\bigr\rangle\ +\ a\,\bigl\langle\,U\,,\,\Dd_{\,h}^{\,2}\scal U\,\bigr\rangle\ +\ a\,\bigl\langle\,V\,,\,\Dd_{\,h}^{\,2}\scal V\,\bigr\rangle\Bigr)\\ 
  +\ \bigl\langle\,G_{\,h}\,(U\,,\,V),\,{e}_{\,h}\,\bigr\rangle\,,
\end{multline}
where
\begin{equation*}
  G_{\,h}\,(U,\,V)\ \eqdef\ \frac{\beta_{\,1\,1}}{3}\;U^{\,3}\ +\ \frac{\beta_{\,1\,2}}{2}\;U^{\,2}\,V\ +\ \beta_{\,2\,2}\,U\,V^{\,2}\ +\ \frac{\alpha_{\,2\,2}}{3}\;V^{\,3}\,,
\end{equation*}
with the products in the \textsc{Hadamard} sense.

The preservation of discrete versions \eqref{eq:lineard} of the linear invariants \eqref{eq:globalq6} holds when the grid operator $\Dd_{\,h}$ satisfies
\begin{equation}\label{eq:dkernel1}
  \Dd_{\,h}^{\,\top}\scal e_{\,h}\ =\ 0\,,
\end{equation}
(\cf \eqref{eq:kernel}, the proof is similar to that of Section~\ref{sec:sec3}). Finally, in the same way as for $\Ic_{\,h}$ in Remark~\ref{r9}, now the semi-discrete version $h\,\I_{\,h}$ of the quadratic invariant \eqref{eq:globalq5}, where
\begin{equation*}
  \I_{\,h}\,(t)\ \eqdef\ \bigl\langle\,\eta_{\,h}\,(t)\,,\,u_{\,h}\,(t)\,\bigr\rangle\ +\ b\,\bigl\langle\,\Dd_{\,h}\scal\eta_{\,h}\,(t)\,,\, \Dd_{\,h}\scal u_{\,h}\,(t)\,\bigr\rangle\,,
\end{equation*}
satisfies
\begin{eqnarray*}
  \od{\I_{\,h}}{t}\ =\ \bigl\langle\,\A_{\,h}\,(\eta_{\,h}\,,\,u_{\,h})\,,\,\Dd_{\,h}\scal u_{\,h}\,\bigr\rangle\ +\ \bigl\langle\,\B_{\,h}\,(\eta_{\,h}\,,\,u_{\,h})\,,\Dd_{\,h}\scal\eta_{\,h}\,\bigr\rangle\,, \label{eq:order2}
\end{eqnarray*}
and similar comments apply; in particular, those concerning the preservation when approximating symmetric solutions, like solitary waves.
\end{remark}

%%% ----------------------------------------------------------------------- %%%

\subsection{Full discretization with symplectic methods}

The formulation of a full MS discretization from \eqref{eq:msp22} is performed by using a symplectic method as time integrator. The resulting scheme will preserve by construction the corresponding discrete version of the semi-discrete conservation law \eqref{eq:sdcl}. This will be here analyzed by taking the IMR as a case study. The corresponding full discretization of \eqref{eq:msp11} has the form
\begin{equation}\label{eq:msf11}
  \K\scal\Dd_{\,t}\,\z_{\,j}^{\,n}\ +\ \M\scal\Mm_{\,t}\,\bigl(\Cc_{\,h}\scal\z^{\,n}\bigr)_{\,j}\ =\ \grad_{\,\z}\,\S\,\bigl(\Mm_{\,t}\scal\z_{\,j}^{\,n}\bigr)\,,
\end{equation}
where $\z^{\,n}\ =\ (\z_{\,j}^{\,n})_{\,j\,\in\,\Z}\ \in\ S_{\,h}$ with $\z_{\,j}^{\,n}\ \in\ \R^{\,d}$ approximating $z_{\,h,\,j}\,(t_{\,n})\,$, $t_{\,n}\ =\ n\,\Delta\,$, $n\ =\ 0,\,1,\,\ldots$ and $\Dd_{\,t}\,$, $\Mm_{\,t}$ are given in \eqref{eq:msopb}. In compact form, \eqref{eq:msf11} reads
\begin{equation}\label{eq:msf22}
  \bigl(\Id_{\,N}\,\otimes\,\K\bigr)\scal\frac{\z^{\,n\,+\,1}\ -\ \z^{\,n}}{\Delta\,t}\ +\ \bigl(\Dd_{\,h}\,\otimes\,\M\bigr)\scal\z^{\,n\,+\,1/2}\ =\ \grad\,\,\S\,(\z^{\,n\,+\,1/2})\,,
\end{equation}
where $\z^{\,n\,+\,1/2}\ \eqdef\ (\Id_{\,N}\,\otimes\,\Mm_{\,t})\scal\z^{\,n}$ and the right hand side of \eqref{eq:msf22} stands for the $d\cdot N$ vector with components $\grad_{\,\z}\,\S\,\bigl((\z^{\,n\,+\,1/2})_{\,j}\bigr)\,$, $j\ =\ 0,\,\ldots,\,N\,-\,1\,$. In the case of \eqref{eq:5s}, \eqref{eq:6s} with $d\ =\ 10$ and the elements of the MS formulation given by \eqref{eq:K} -- \eqref{eq:pot}, the simplified version reads
\begin{multline}
  \bigl(\,\Id_{\,N}\ -\ b\,\Dd_{\,h}^{\,2}\,\bigr)\scal\frac{\eta^{\,n\,+\,1}\ -\ \eta^{\,n}}{\Delta t}\ +\ \Dd_{\,h}\scal\bigl(\,\Id_{\,N}\ +\ a\,\Dd_{\,h}^{\,2}\,\bigr)\scal u^{\,n\,+\,1/2} \\ 
  +\ \Dd_{\,h}\scal\A_{\,h}\,(\eta^{\,n\,+\,1/2}\,,\,u^{\,n\,+\,1/2})\ =\ 0\,, \label{eq:fdsimply1}
\end{multline}
\begin{multline}
  \bigl(\,\Id_{\,N}\ -\ d\,\Dd_{\,h}^{\,2}\,\bigr)\scal\frac{u^{\,n\,+\,1}\ -\ u^{\,n}}{\Delta t}\ +\ \Dd_{\,h}\scal\bigl(\,\Id_{\,N}\ +\ a\,\Dd_{\,h}^{\,2}\,\bigr)\scal\eta^{\,n\,+\,1/2} \\ 
  +\ \Dd_{\,h}\scal\B_{\,h}\,(\eta^{\,n\,+\,1/2}\,,\,u^{\,n\,+\,1/2})\ =\ 0\,. \label{eq:fdsimply2}
\end{multline}
where $\eta^{\,n}\ =\ (\eta_{\,0}^{\,n},\,\ldots,\,\eta_{\,N\,-\,1}^{\,n})^{\,\top}\,$, $u^{\,n}\ =\ (u_{\,0}^{\,n},\,\ldots,\,u_{\,N\,-\,1}^{\,n})^{\,\top}\,$.

As usual, in terms of ${\z}_{\,j}^{\,n\,+\,1/2}\,$, Equation~\eqref{eq:msf22} becomes a fixed point system for each step of the form
\begin{equation}\label{eq:fps}
  \Bigl[\bigl(\,\Id_{\,N}\,\otimes\,\K\,\bigr)\ +\ \frac{\Delta t}{2}\;\bigl(\Dd_{\,h}\,\otimes\,\M\,\bigr)\Bigr]\scal Z\ =\ \bigl(\,\Id_{\,N}\,\otimes\,\K,\bigr)\scal\z^{\,n}\ +\ \frac{\Delta t}{2}\;\grad_{\,\z}\,\S\,(Z)\,,
\end{equation}
with $\z^{\,n\,+\,1}\ =\ 2\,Z\ -\ {\z}^{\,n}\,$. The system \eqref{eq:fps} can be iteratively solved by dividing the gradient term
\begin{equation*}
  \grad_{\,\z}\,\S\,(\z)\ =\ \L\,(\z)\ +\ \N\,(\z)\,,
\end{equation*}
into the linear ($\L$) and the nonlinear ($\N$) parts and solving the following modified version of the fixed point algorithm
\begin{multline*}%\label{eq:fps2}
  \Bigl[(\Id_{\,N}\,\otimes\,\K)\ +\ \frac{\Delta t}{2}\;(\Dd_{\,h}\,\otimes\,\M)\ -\ \frac{\Delta t}{2}\;(\Id_{\,N}\,\otimes\,\L)\Bigr]\scal Z^{\,[\,\nu\,+\,1]} \\ 
  =\ (\Id_{\,N}\,\otimes\,\K)\scal\z^{\,n}\ +\ \frac{\Delta t}{2}\;(\Id_{\,N}\,\otimes\,\N)\scal(Z^{\,[\,\nu\,]})\,, 
\end{multline*}
for $\nu\ =\ 0,\,1,\,\ldots$

We now derive the corresponding discrete conservation law. Let $U\,$, $V$ be solutions of the variational equation associated to \eqref{eq:msf22}:
\begin{equation}%\label{eq:fdvareq}
  (\Id_{\,N}\,\otimes\,\K)\scal\frac{U^{\,n\,+\,1}\ -\ U^{\,n}}{\Delta t}\ +\ 
  (\Dd_{\,h}\,\otimes\,\M)\scal U^{\,n\,+\,1/2}\ =\ \S^{\,\prime\prime}\,\bigl(\,\z^{\,n\,+\,1/2}\,\bigr)\,U^{\,n\,+\,1/2}\,, \label{eq:fdvareq1}
\end{equation}
\begin{equation}
  (\Id_{\,N}\,\otimes\,\K)\scal\frac{V^{\,n\,+\,1}\ -\ V^{\,n}}{\Delta t}\ +\ 
  (\Dd_{\,h}\,\otimes\,\M)\scal V^{\,n\,+\,1/2}\ =\ \S^{\,\prime\prime}\,\bigl(\,\z^{\,n\,+\,1/2}\,\bigr)\,V^{\,n\,+\,1/2}\,, \label{eq:fdvareq2}
\end{equation}
Substracting the inner product of \eqref{eq:fdvareq1} with $V^{\,n\,+\,1/2}$ and the inner product of \eqref{eq:fdvareq2} with $U^{\,n\,+\,1/2}\,$, after some simplifications we have
\begin{multline}
  \frac{1}{\Delta t}\;\Bigl(\bigl\langle\,(\Id_{\,N}\,\otimes\,\K)\scal{U^{\,n\,+\,1}}\,,\,V^{\,n\,+\,1}\,\bigr\rangle\ -\ \bigl\langle\,(\Id_{\,N}\,\otimes\,\K)\scal U^{\,n}\,,\,V^{\,n}\,\bigr\rangle\Bigr)\ + \\
  \bigl\langle\,(\Dd_{\,h}\,\otimes\,\M)\scal U^{\,n\,+\,1/2}\,,\,V^{\,n\,+\,1/2}\,\bigr\rangle\ +\ \bigl\langle\,(\Dd_{\,h}^{\,\top}\,\otimes\,\M)\scal U^{\,n\,+\,1/2}\,,\,V^{\,n\,+\,1/2}\,\bigr\rangle\ =\ 0\,, \label{eq:fdcl}
\end{multline}
which is the fully discrete version of the MS conservation law satisfied by \eqref{eq:msf22}. Observe that if $\Dd_{\,h}^{\,\top}\ =\ -\,\Dd_{\,h}\,$, then the last two terms in \eqref{eq:fdcl} are cancelled, leading to the fully discrete version of \eqref{eq:tsymp}.

Using some results of previous sections additional properties of \eqref{eq:msf22}, related to the MS and \textsc{Hamiltonian} formulations are summarized in the following sections.

%%% ----------------------------------------------------------------------- %%%

\subsubsection{Numerical dispersion relation}

The numerical dispersion relation for \eqref{eq:msf22} makes use of the results obtained in the semi-discrete case and several references, \cite{Frank2006, McLachlan2014}. Thus, for the linearized system
\begin{equation*}
  \left(\Id_{\,N}\,\otimes\,\K\right)\scal\frac{\z^{\,n\,+\,1}\ -\ \z^{\,n}}{\Delta t}\ +\ \left(\Dd_{\,h}\,\otimes\,\M\right)\scal\z^{\,n\,+\,1/2}\ =\ \L\,(\z^{\,n\,+\,1/2})\,,
\end{equation*}
the ansatz of a periodic solution
\begin{equation*}
  \z_{\,h,\,j}\,(t)\ =\ \ue^{\,\ui\,(\Omega\,n\,\Delta t\ +\ \xi\,j\,h)}\,a\,, \qquad a\ \in\ \R^{\,d}\,,
\end{equation*}
with $a$ solution of \eqref{eq:dispr1b} and $\xi$ satisfying \eqref{soldisp2}, leads to the dispersion relations
\begin{equation*}
  \frac{\ue^{\,\ui\,\Omega\,\Delta t}\ -\ 1}{\Delta t}\ =\ \ui\,\omega\;\frac{\ue^{\ui\,\Omega\,\Delta t}\ +\ 1}{2}\,, \qquad (\Dd_{\,h}\ -\ \ui\,k)\scal\ep\ =\ 0\,,
\end{equation*}
where the first equation can be alternatively written as
\begin{equation}\label{eq:imrdl}
  \omega\ =\ \frac{2}{\Delta t}\;\tan\,\Bigl(\,\frac{\Omega\,\Delta t}{2}\,\Bigr)\,,
\end{equation}
and $k\ =\ k\,(\xi)$ follows  the approach developed in Section~\ref{sec:412} and for $\omega\,$, $k$ satisfying the continuous dispersion relation. Formula \eqref{eq:imrdl} typically corresponds to the IMR as time integrator; see \cite{Ascher2004, Frank2006, McLachlan2014} for the derivation of the corresponding formulas when any other method of the GLRK family is used.

%%% ----------------------------------------------------------------------- %%%

\subsubsection{Conservation properties}
\label{sec:422}

We focus now on the behaviour of the type of full discretization represented by \eqref{eq:msf22} with respect to the global quantities of the associated semi-discrete problem. Standard theory of \textsc{Runge}--\textsc{Kutta} (RK) methods \cite{Hairer2002}, establishes the preservation of the linear invariants \eqref{eq:lineard}
\begin{equation*}
  C_{\,1,\,h}\,(\,\eta^{\,n}\,,\,u^{\,n}\,)\ =\ C_{\,1,\,h}\,(\,\eta^{\,0}\,,\,u^{\,0}\,)\,, \qquad
  C_{\,2,\,h}\,(\,\eta^{\,n}\,,\,u^{\,n}\,)\ =\ C_{\,2,\,h}\,(\,\eta^{\,0}\,,\,u^{\,0}\,)\,,
\end{equation*}
where $(\eta^{\,n}\,,\,u^{\,n})$ is a solution of \eqref{eq:fdsimply1}, \eqref{eq:fdsimply2} with initial condition $(\eta^{\,0}\,,\,u^{\,0})\,$. As for the quadratic quantities $\Ic_{\,h}\,$, $\I_{\,h}\,$, under the symmetric conditions explained in Remark~\ref{r9}, the preservation also holds by the numerical solution given by any symplectic method, \cf \cite{Cano2006}. This is illustrated in Figures~\ref{Fig:swp1} and \ref{Fig:swp2}. The scheme \eqref{eq:msf22} was implemented on the segment $[\,-256,\,256\,]$ and with $h\ =\ 0.125\,$. The operator $\Dd_{\,h}$ is given by the pseudo-spectral differentiation operator to simulate classical (CSW) and generalized (GSW) solitary-wave solutions of \eqref{eq:5s}, \eqref{eq:6s} in cases with MS and \textsc{Hamiltonian} structures and with the MS structure only. They show the approximate $\eta$ profile of the solitary wave and the time behaviour of the error of the corresponding quantity and different time step sizes. The results confirm the preservation when simulating symmetric solutions like these profiles.

\begin{figure}
  \centering
  \subfigure[]{\includegraphics[width=7.5cm]{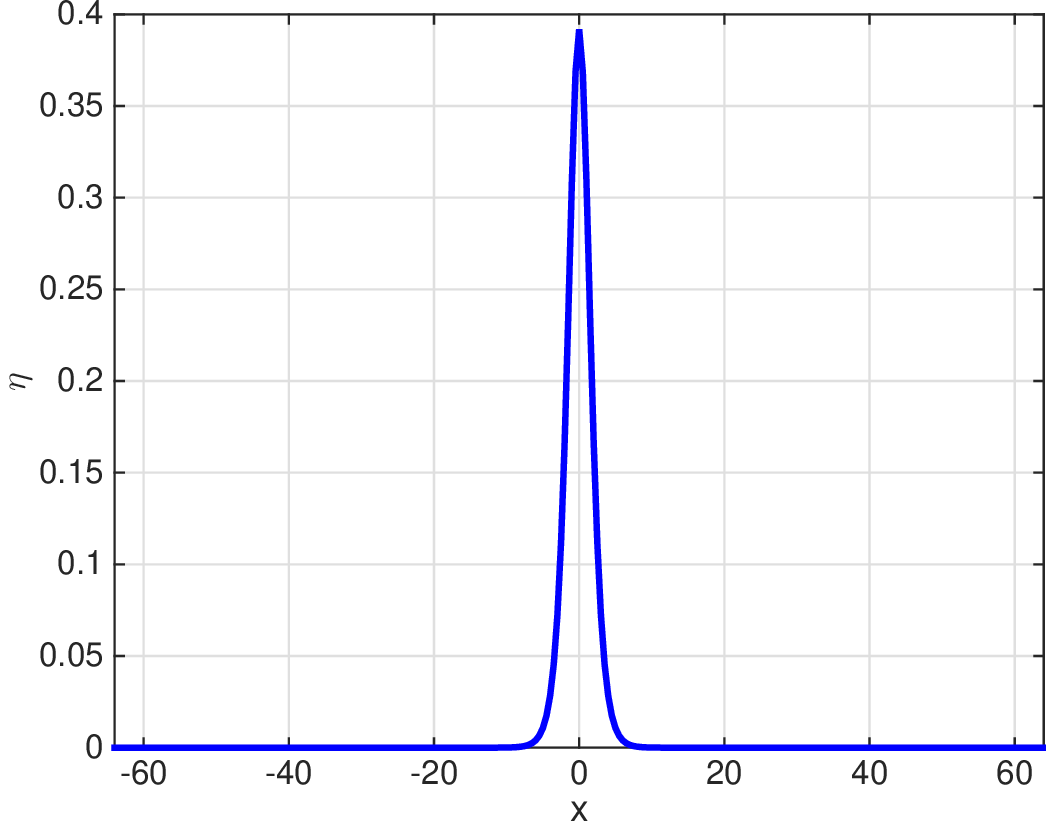}}
  \subfigure[]{\includegraphics[width=7.5cm]{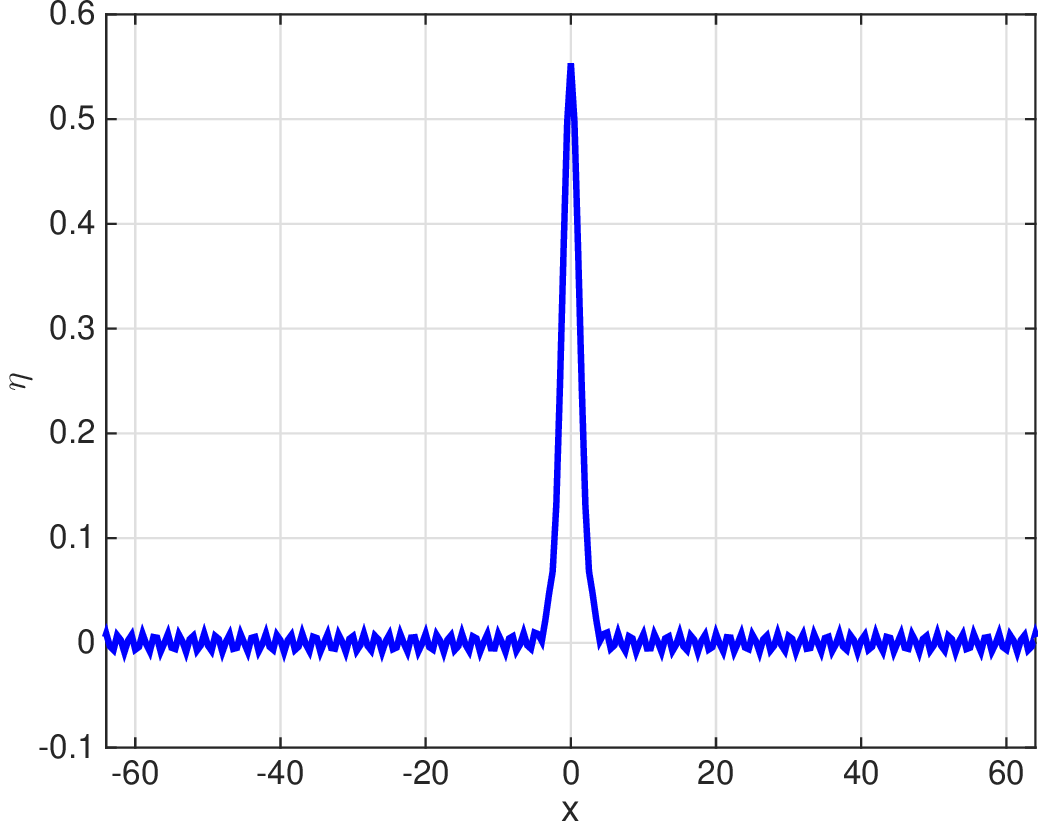}}
  \subfigure[]{\includegraphics[width=7.5cm]{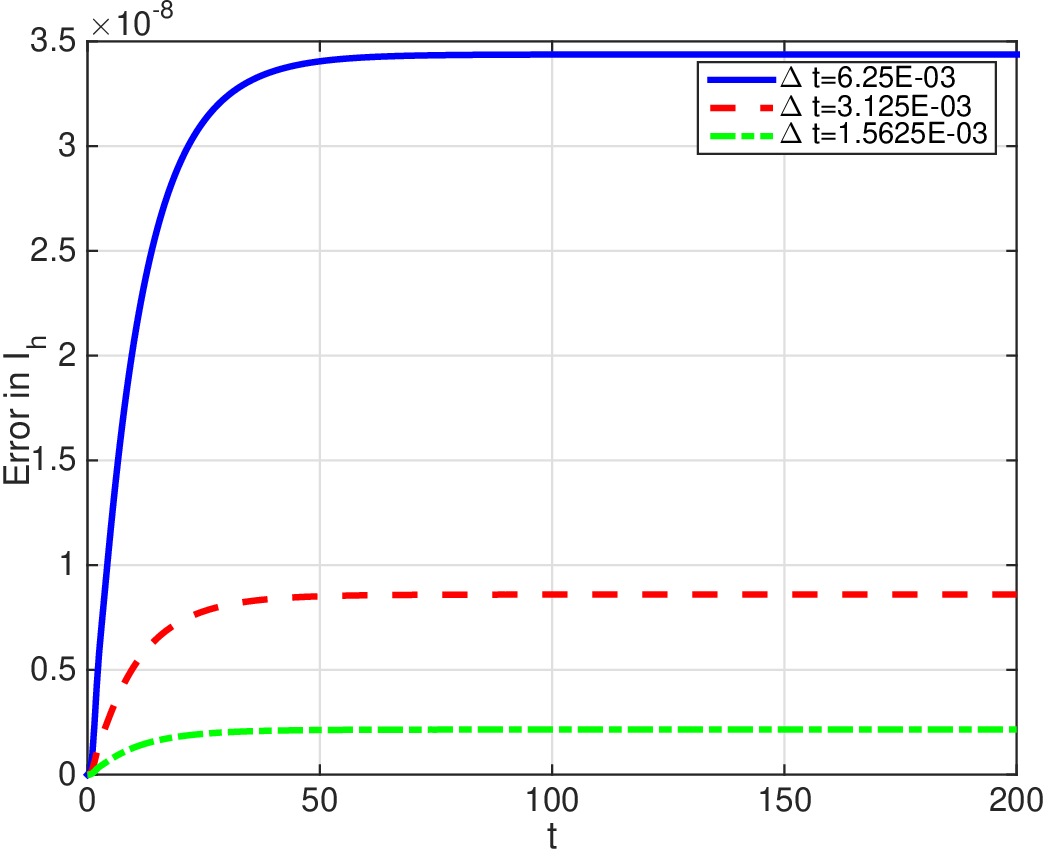}}
  \subfigure[]{\includegraphics[width=7.5cm]{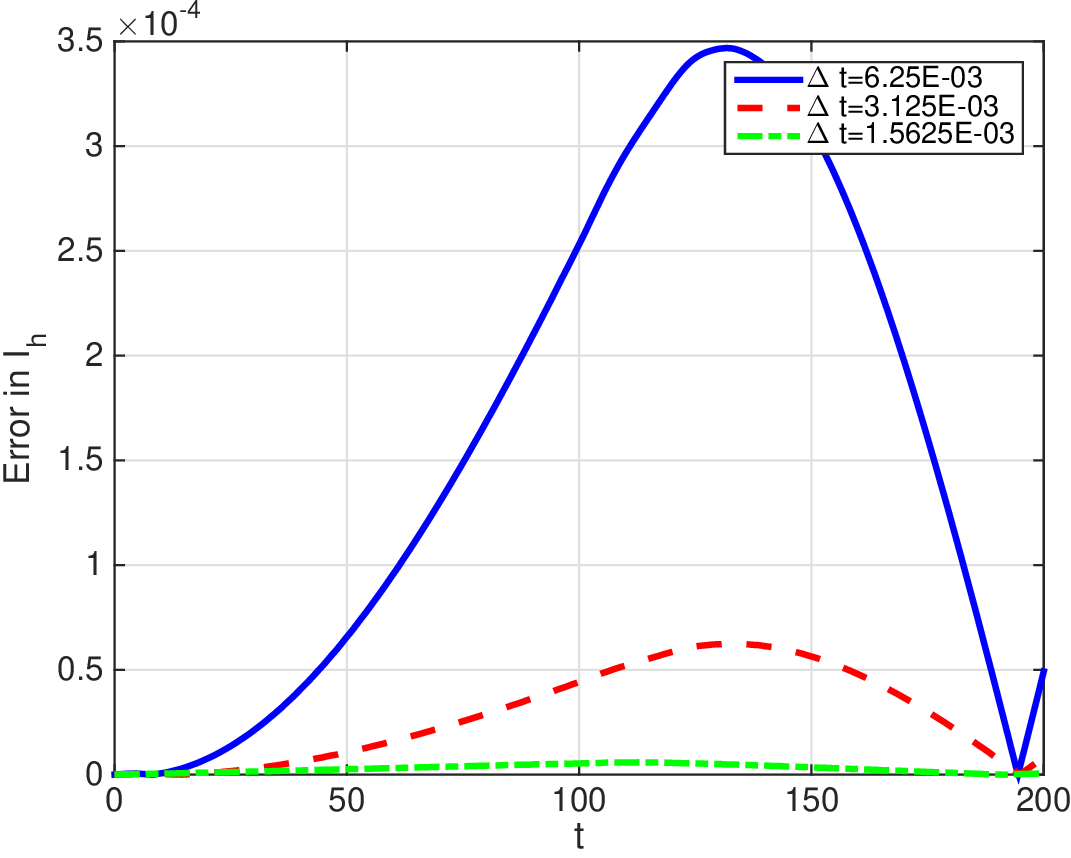}}
  \caption{Numerical approximation of Equations~\eqref{eq:5s}, \eqref{eq:6s} with $h\ =\ 0.125\,$, $\alpha_{\,1\,1}\ =\ 0\,$, $\alpha_{\,1\,2}\ =\ 0.46\,$, $\alpha_{\,2\,2}\ =\ 0\,$, $\beta_{\,1\,1}\ =\ 0.23\,$, $\beta_{\,1\,2}\ =\ 0\,$, $\beta_{\,2\,2}\ =\ 0.73$ by \eqref{eq:msf22}. (a), (c) $\zeta$ CSW profile for $a\ =\ c\ =\ 0\,$, $b\ =\ d\ =\ 1/6\,$, $c_{\,s}\ =\ 1.2$ and error in the invariant $\I_{\,h}$ vs. time; (b), (d) $\zeta$ GSW profile for $a\ =\ c\ =\ 1/6\,$, $b\ =\ d\ =\ 0\,$, $c_{\,s}\ =\ 1.3$ and error in the invariant  $\I_{\,h}$ vs. time.}
  \label{Fig:swp1}
\end{figure}

\begin{figure}
  \centering
  \vspace*{-1.0em}
  \subfigure[]{\includegraphics[width=0.4\textwidth]{figs/csw3.eps}}
  \vspace*{-2.75em}
  \subfigure[]{\includegraphics[width=0.99\textwidth]{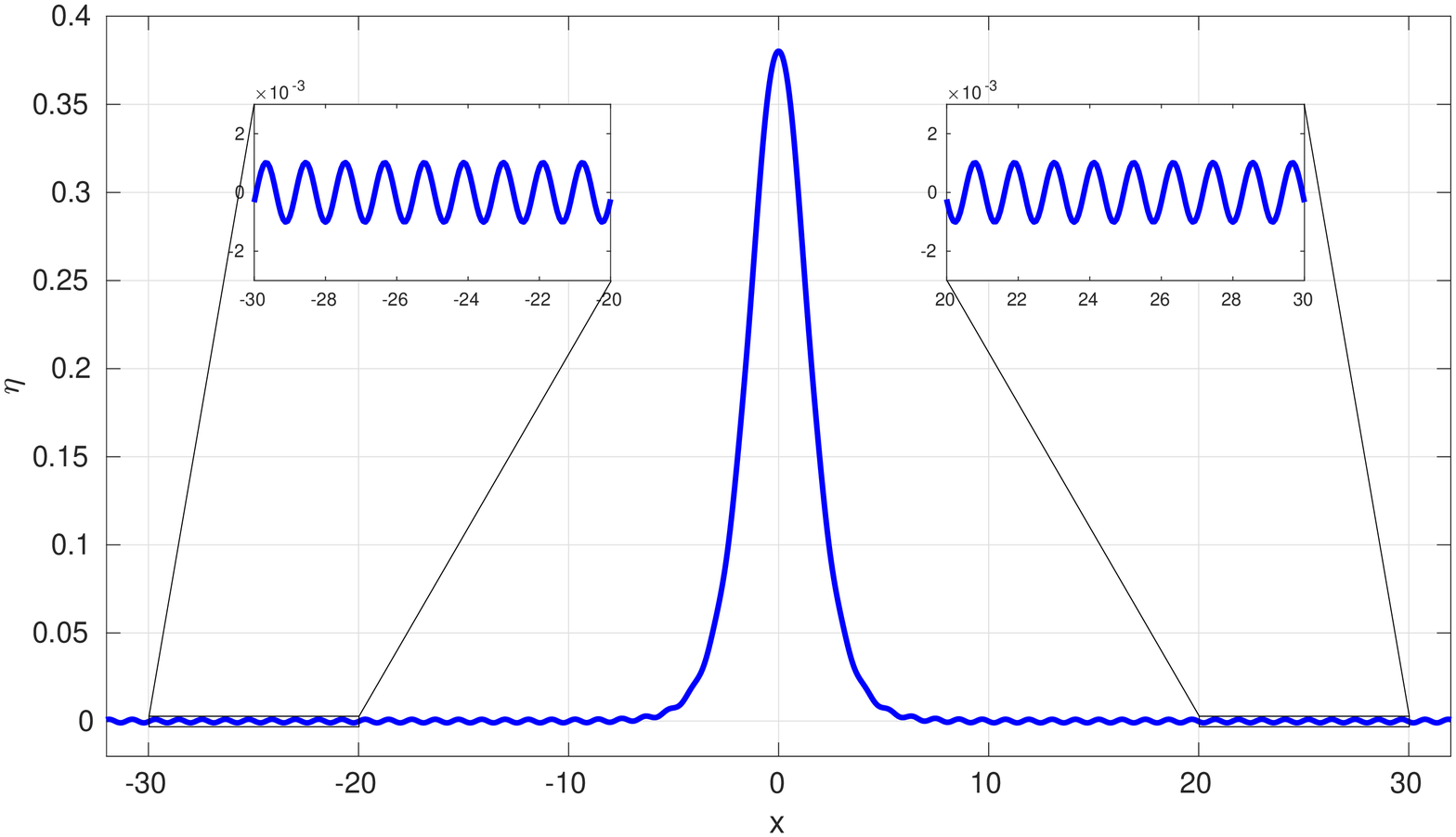}}
  \vspace*{-0.15em}
  \subfigure[]{\includegraphics[width=0.4\textwidth]{figs/ms2_fig3.eps}}
  \subfigure[]{\includegraphics[width=0.4\textwidth]{figs/ms2_fig4.eps}}
  \vspace*{-0.5em}
  \caption{Numerical approximation of Equations~\eqref{eq:5s}, \eqref{eq:6s} with $h\ =\ 0.125\,$, $\alpha_{\,1\,1}\ =\ 0\,$, $\alpha_{\,1\,2}\ =\ 0.46\,$, $\alpha_{\,2\,2}\ =\ 0\,$, $\beta_{\,1\,1}\ =\ 0.23\,$, $\beta_{\,1\,2}\ =\ 0\,$, $\beta_{\,2\,2}\ =\ 0.73$ by \eqref{eq:msf22}. (a), (c) $\zeta$ CSW profile for $a\ =\ c\ =\ 0\,$, $b\ =\ 1/4\,$, $d\ =\ 1/12\,$, $c_{\,s}\ =\ 1.05$ and error in the invariant  $\Ic_{h}$ vs. time; (b), (d) $\zeta$ GSW profile for $a\ =\ c\ =\ 1/9\,$, $b\ =\ 1/9\,$, $d\ =\ 0\,$, $c_{\,s}\ =\ 1.2$ and error in the invariant $\Ic_{\,h}$ vs. time.}
  \label{Fig:swp2}
\end{figure}

As far as the energy quantities $\Ec_{\,h}$ and $\H_{\,h}$ are concerned, the good behaviour when approximating solitary waves is expected, due to the relative equilibrium condition \eqref{eq:msre}. This behaviour is observed  in Figures~\ref{Fig:swp3} and \ref{Fig:swp4}. They display the error in the energy and the \textsc{Hamiltonian} as functions of time when simulating the solitary waves solutions considered in Figures~\ref{Fig:swp1} and \ref{Fig:swp2}.

\begin{figure}
  \centering
  \subfigure[]{\includegraphics[width=7.5cm]{figs/csw1.eps}}
  \subfigure[]{\includegraphics[width=7.5cm]{figs/gsw2.eps}}
  \subfigure[]{\includegraphics[width=7.5cm]{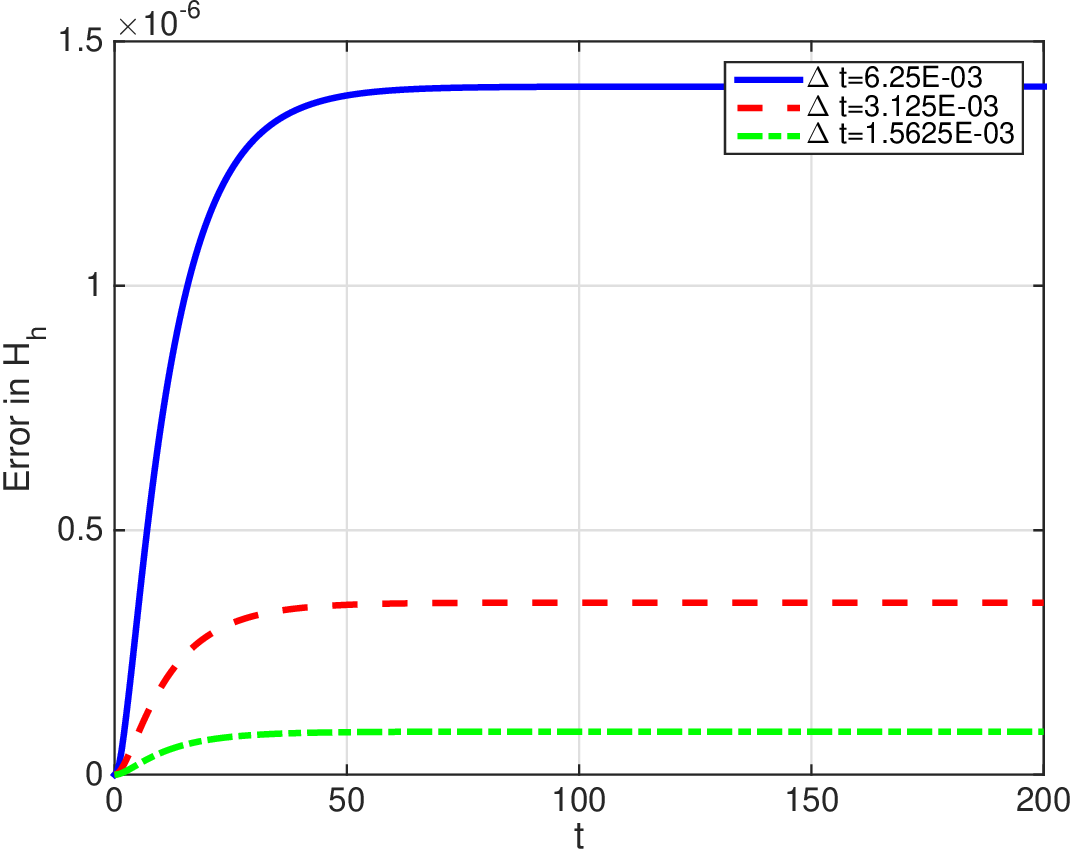}}
  \subfigure[]{\includegraphics[width=7.5cm]{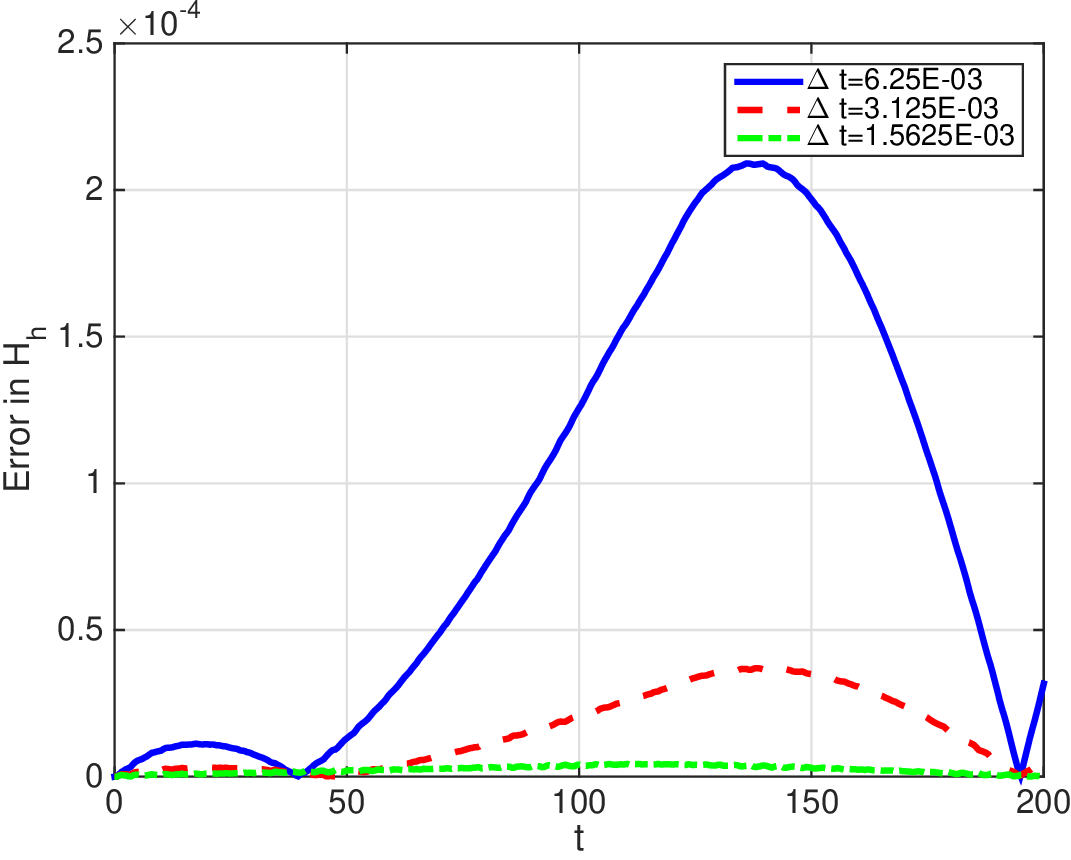}}
  \caption{Numerical approximation of Equations~\eqref{eq:5s}, \eqref{eq:6s} with $h\ =\ 0.125\,$, $\alpha_{\,1\,1}\ =\ 0\,$, $\alpha_{\,1\,2}\ =\ 0.46\,$, $\alpha_{\,2\,2}\ =\ 0\,$, $\beta_{\,1\,1}\ =\ 0.23\,$, $\beta_{\,1\,2}\ =\ 0\,$, $\beta_{\,2\,2}\ =\ 0.73$ by \eqref{eq:msf22}. (a), (c) $\zeta$ CSW profile for $a\ =\ c\ =\ 0\,$, $b\ =\ d\ =\ 1/6\,$, $c_{\,s}\ =\ 1.2$ and error in the invariant $\H_{h}$ vs. time; (b), (d) $\zeta$ GSW profile for $a\ =\ c\ =\ 1/6\,$, $b\ =\ d\ =\ 0\,$, $c_{\,s}\ =\ 1.3$ and error in the invariant $\H_{\,h}$ vs. time.}
  \label{Fig:swp3}
\end{figure}

\begin{figure}
  \centering
  \vspace*{-1.0em}
  \subfigure[]{\includegraphics[width=0.4\textwidth]{figs/csw3.eps}}
  \vspace*{-2.75em}
  \subfigure[]{\includegraphics[width=0.99\textwidth]{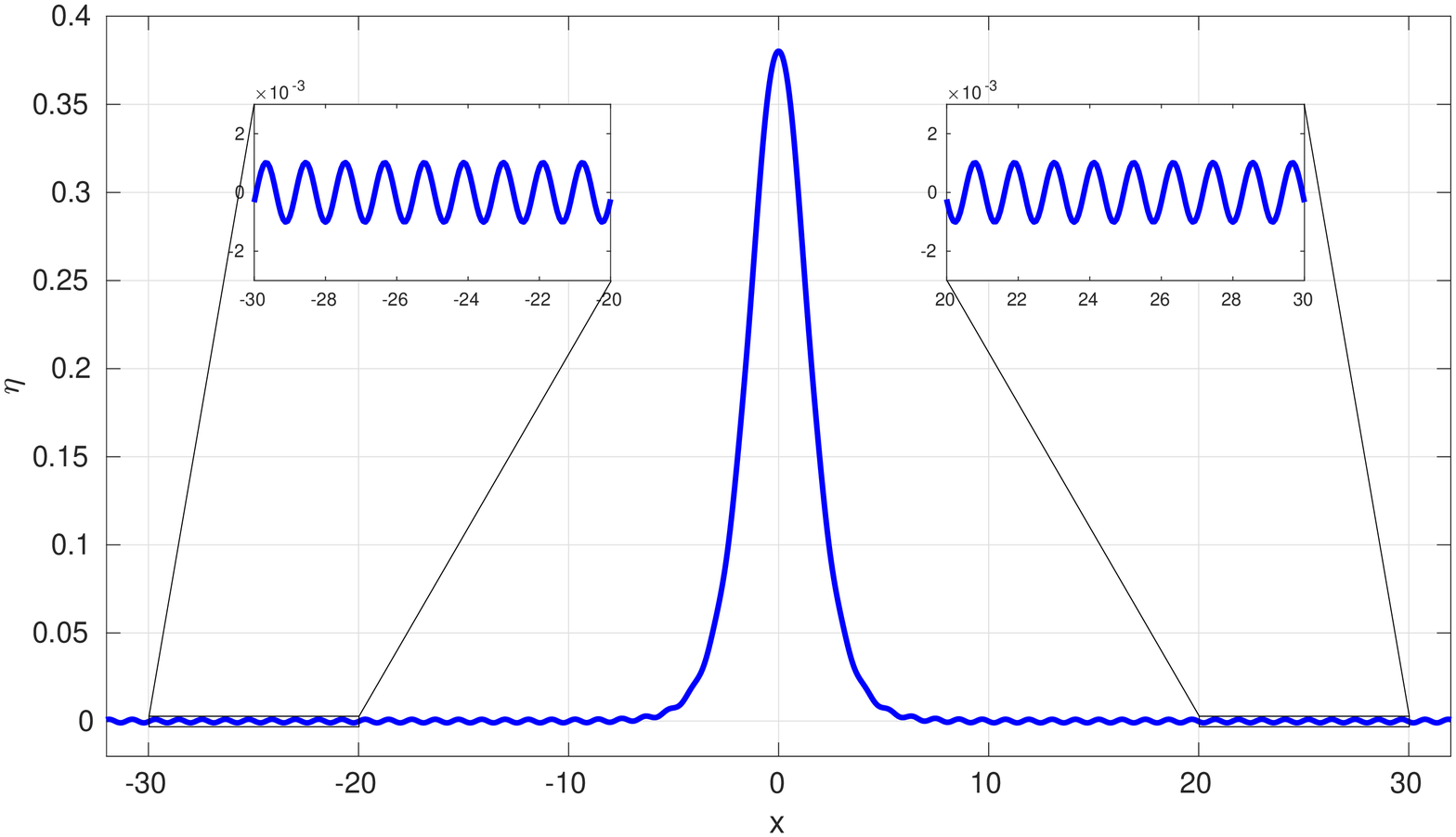}}
  \vspace*{-0.15em}
  \subfigure[]{\includegraphics[width=7.5cm]{figs/ms2_fig7.eps}}
  \subfigure[]{\includegraphics[width=7.5cm]{figs/ms2_fig8.eps}}
  \vspace*{-0.5em}
  \caption{Numerical approximation of Equations~\eqref{eq:5s}, \eqref{eq:6s} with $h\ =\ 0.125\,$, $\alpha_{\,1\,1}\ =\ 0\,$, $\alpha_{\,1\,2}\ =\ 0.46\,$, $\alpha_{\,2\,2}\ =\ 0\,$, $\beta_{\,1\,1}\ =\ 0.23\,$, $\beta_{\,1\,2}\ =\ 0\,$, $\beta_{\,2\,2}\ =\ 0.73$ by \eqref{eq:msf22}. (a), (c) $\zeta$ CSW profile for $a\ =\ c\ =\ 0\,$, $b\ =\ 1/4\,$, $d\ =\ 1/12\,$, $c_{\,s}\ =\ 1.05$ and error in the invariant  $\Ec_{\,h}$ vs. time; (b), (d) $\zeta$ GSW profile for $a\ =\ c\ =\ 1/9\,$, $b\ =\ 1/9\,$, $d\ =\ 0\,$, $c_{\,s}\ =\ 1.2$ and error in the invariant $\Ec_{\,h}$ vs. time.}
  \label{Fig:swp4}
\end{figure}

%%% ----------------------------------------------------------------------- %%%

\section{Concluding remarks}
\label{sec:sec5}

In the present paper the numerical approximation of the periodic IVP of some systems of \textsc{Boussinesq} type, proposed as models for surface wave propagation, is discussed. The analysis is focused on the preservation, in some numerical sense, of the multi-symplectic and \textsc{Hamiltonian} structures of some of the equations, studied previously by the authors in a recent paper \cite{Duran2019}, as well as some consequences of these structures, such as the linear dispersion relation and the behaviour with respect to local and global invariant quantities.

Following the literature on the subject, two approaches are considered. The first one is based on constructing MS discretizations by integrating, in both space and time, with symplectic methods, \cite{Ascher2004, Ascher2005, Bridges2001}. Some of the properties and drawbacks of this classical approach, discussed here, are obtained from direct application of previous, general results on this matter, \cite{Bridges2001, Bridges2001a, Frank2006, Reich2000a, McLachlan2014}, to the \textsc{Boussinesq}-type equations under study. Among them, we confirm the difficulties to define an explicit ODE system when the equations are discretized in space with GLRK methods, \cite{Reich2000a, McLachlan2014}. As a first attempt to overcome this main drawback, we also study the application of PRK methods into this approach, recently proposed in \cite{Ryland2008}. We find that our systems cannot be written in a \textsc{Darboux} normal form which leads to a well-defined ODE system when using semi-discretizations in space based on the family of \textsc{Lobatto} IIIA-IIIB PRK methods. We still explore these schemes directly, using different partitions of the variables, with no improvements with respect to the GLRK methods and leaving this point as an open question for future research.

The origin of the second approach considered in this paper for the numerical approximation of the \textsc{Boussinesq}-type equations is in accordance to the strategy adopted in some references, \cite{Bridges2001a, Chen2001, Islas2003a, Islas2006}, where the corresponding equations are discretized in space with spectral methods and in time with a symplectic integrator. Generalizing this idea, we introduce semi-discretizations in space based on a general grid operator approximating the spatial partial derivative and discuss the requirements on it leading to a explicit ODE semi-discrete system, as well as semi-discrete versions of the MS conservation law. The linear dispersion relation and the preservation of local and global quantities of the periodic IVP are also studied. The approach is developed for a general MS system and specific results for the \textsc{Boussinesq}-type equations under study are emphasized. The semi-discrete system is then integrated in time with a symplectic method. For the present paper, the implicit midpoint rule is chosen; the corresponding fully discrete MS conservation law and the linear dispersion relation are derived, while the behaviour with respect to the semi-discrete global quantities is discussed, some analytically and some computationally, emphasizing the good behaviour when simulating solitary-wave solutions. These results can be adapted in a relatively direct way if other symplectic methods, such as another one of the GLRK family or any composition method based on the implicit midpoint rule, is used. We leave this generalization to future works.

%%% ----------------------------------------------------------------------- %%%

\bigskip\bigskip
\subsection*{Acknowledgments}
\addcontentsline{toc}{subsection}{Acknowledgments}

AD was supported by Junta de \textsc{Castilla y Leon} and Fondos FEDER under the Grant VA041P17. DM's work was supported by the \textsc{Marsden} Fund administered by the Royal Society of \textsc{New Zealand} with contract number VUW1418. AD and DM would like to acknowledge the support from the University \textsc{Savoie Mont Blanc} and the hospitality of LAMA UMR \#5127 during their respective visits in 2019.

%%% ----------------------------------------------------------------------- %%%

\appendix
\section{Abbreviations}

\begin{description}
  \item[BEA] Backward Error Analysis
  \item[CSW] Classical Solitary Wave
  \item[IMR] Implicit Midpoint Rule
  \item[IVP] Initial-Valuer Problem
  \item[GLRK] \textsc{Gau\ss}--\textsc{Legendre} \textsc{Runge}--\textsc{Kutta}
  \item[GSW] Generalized Solitary Wave
  \item[KdV] \textsc{Korteweg}--\textsc{de Vries}
  \item[MS] Multi-Symplectic
  \item[NLS] Nonlinear \textsc{Schr\"{o}dinger}
  \item[ODE] Ordinary Differential Equation
  \item[PDE] Partial Differential Equation
  \item[PRK] Partitioned \textsc{Runge}--\textsc{Kutta}
\end{description}

%%% ----------------------------------------------------------------------- %%%

%%% Bibliography
\bigskip\bigskip
\addcontentsline{toc}{section}{References}
\bibliographystyle{abbrv}
\bibliography{biblio}
\bigskip\bigskip

\end{document}